\documentclass[leqno,a4paper]{article}

\usepackage[svgnames,x11names]{xcolor}
\usepackage{hyperref}
  \hypersetup{
    colorlinks, linkcolor={Chocolate4},
    citecolor={Chocolate4}, urlcolor={DarkOliveGreen4}
  }
\usepackage{amsmath,amsthm,amssymb}
\usepackage[utf8]{inputenc}
\usepackage[T1]{fontenc}

\usepackage{enumerate}
\usepackage{enumitem}
\usepackage{bbm}
\usepackage{todonotes}

\usepackage[final]{changes}
\definechangesauthor[name={El Houcein}, color=red]{EH}
\definechangesauthor[name={Joanna}, color=orange]{J}
\definechangesauthor[name={Mariusz}, color=green]{M}
\definechangesauthor[name={Thierry}, color=cyan]{T}

\usepackage{mathrsfs}

\newtheorem{Th}{Theorem}
\newtheorem{Prop}[Th]{Proposition}
\newtheorem{Lemma}[Th]{Lemma}
\newtheorem{Cor}[Th]{Corollary}
\newtheorem{Def}[Th]{Definition}

\theoremstyle{definition}
\newtheorem{Remark}[Th]{Remark}
\newtheorem{Question}[Th]{Question}

\theoremstyle{plain}
\newcounter{MainTheoremCounter}

\theoremstyle{plain}
\newtheorem*{namedthm}{\namedthmname}
\newcounter{namedthm}
\makeatletter
	\newenvironment{named}[2]
	{\def\namedthmname{#1}
	\refstepcounter{namedthm}
	\namedthm[#2]\def\@currentlabel{#1}}
	{\endnamedthm}
\makeatother

	\newtheorem{property}{Property}
	\makeatletter
	\newcommand{\setpropertytag}[1]{
	\let\oldtheproperty\theproperty
	\renewcommand{\theproperty}{#1}
	\g@addto@macro\endproperty{
	\global\let\theproperty\oldtheproperty}
	\makeatother
  }

\newcommand{\beq}{\begin{equation}}
\newcommand{\eeq}{\end{equation}}

\def\scalar(#1,#2){(#1\mid#2)}

\newcommand{\cs}{{\cal S}}

\newcommand{\cb}{{\cal B}}

\newcommand{\cd}{{\cal D}}

\newcommand{\xbm}{(X,{\cal B},\mu)}

\newcommand{\ycn}{(Y,{\cal C},\nu)}
\newcommand{\ot}{\otimes}
\newcommand{\ov}{\overline}
\newcommand{\la}{\lambda}

\newcommand{\bs}{\mathbb{S}}
\newcommand{\A}{\mathbb{A}}

\newcommand{\R}{{\mathbb{R}}}
\newcommand{\T}{{\mathbb{T}}}
\newcommand{\C}{{\mathbb{C}}}
\newcommand{\D}{{\mathbb{D}}}
\newcommand{\Z}{{\mathbb{Z}}}
\newcommand{\N}{{\mathbb{N}}}

\newcommand{\vep}{\varepsilon}
\newcommand{\va}{\varphi}

\newcommand{\mob}{\boldsymbol{\mu}}
\newcommand{\bfu}{\boldsymbol{u}}
\newcommand{\bfv}{\boldsymbol{v}}
\newcommand{\lio}{\boldsymbol{\lambda}}

\newcommand{\tend}[3][]{\xrightarrow[#2\to#3]{#1}}

\usepackage{nicefrac}

	\newcommand{\ind}[1]{\mathbbmss{1}_{#1}} 

\providecommand{\noopsort}[1]{} 

\title{M\"obius disjointness for models of an ergodic system and beyond}
\author{H.\ El Abdalaoui\thanks{ Research supported by the special program in the framework of the Jean Morlet semester ``Ergodic Theory and Dynamical Systems
in their Interactions with Arithmetic and Combinatorics''.}\ , J.\ Ku\l aga-Przymus$^\ast$\thanks{Research supported by Narodowe Centrum Nauki grant  UMO-2014/15/B/ST1/03736.}\ , M.\ Lema\'nczyk$^{\ast\dag}$, T.\ de la Rue$^\ast$}

\begin{document}
\bibliographystyle{amsplaininitials}
\maketitle \normalsize

\thispagestyle{empty}

\begin{abstract} Given a topological dynamical system $(X,T)$ and an arithmetic function $\bfu\colon\N\to\C$, we study the strong MOMO property (relatively to $\bfu$) which is a strong version of  $\bfu$-disjointness with all observable sequences in $(X,T)$. It is proved that, given an ergodic measure-preserving system $(Z,\cd,\kappa,R)$, the strong MOMO property (relatively to $\bfu$) of  a uniquely ergodic model $(X,T)$ of $R$ yields all other uniquely ergodic models of $R$ to be $\bfu$-disjoint.  It follows that  all uniquely ergodic models of: ergodic unipotent diffeomorphisms on nilmanifolds, discrete spectrum automorphisms,  systems given by some substitutions of constant length (including the classical Thue-Morse and Rudin-Shapiro substitutions), systems determined by Kakutani sequences are  M\"obius (and Liouville) disjoint. The validity of Sarnak's conjecture implies the strong MOMO property relatively to $\mob$ in all zero entropy systems, in particular, it makes $\mob$-disjointness uniform. The
absence of strong MOMO property in positive entropy systems is discussed and, it is proved that, under the Chowla conjecture, a topological system
 has the strong MOMO property relatively to the Liouville function if and only if its topological entropy is zero.
\end{abstract}

\tableofcontents

\section{Introduction}

\paragraph{Sarnak's conjecture}
Assume that $T$ is a continuous map of a compact metric space $X$. In 2010, P.~Sarnak \cite{Sa} stated the following conjecture: whenever the (topological) entropy of $T$ is zero,
\beq\label{momoe1}
\frac1N\sum_{n\leq N}f(T^nx)\mob(n)\tend{N}{\infty}0\eeq
for all $f\in C(X)$ and $x\in X$ (we recall that the M\"obius function $\mob$ is defined as $\mob(p_1\ldots p_k)=(-1)^k$ for distinct primes $p_1,\ldots,p_k$, $\mob(1)=1$ and $\mob(n)=0$ for the remaining natural numbers). When~\eqref{momoe1} holds (for all $f\in C(X)$ and $x\in X)$, we also say that the system $(X,T)$ is {\em M\"obius disjoint}.
Sarnak's conjecture is of purely topological dynamics nature. However, measure-theoretic properties of the subshift $X_{\mob}\subset \{-1,0,1\}^\N$ considered with $\kappa$ that is a limit point of
$\frac1N\sum_{n\leq N}\delta_{T^nx}$, $N\geq1$, are often used to determine~\eqref{momoe1}; it is a natural playground where dynamics and number theory meet. One of the most motivating examples of the interplay between them is that the famous Chowla conjecture  for $\mob$~\footnote{ \label{f1} The Chowla conjecture says that $\frac1N\sum_{n\leq N}\mob^{j_0}(n)\mob^{j_1}(n+a_1)\ldots\mob^{j_r}(n+a_r)\tend{N}{\infty}0$ for each choice $1\leq a_1<\ldots<a_r$ and at least one $j_k$ odd; equivalently,  $\mob$ is a generic point for the so-called Sarnak's measure on $\{-1,0,1\}^{\N}$ (\cite{Ab-Ku-Le-Ru, Sa}).} implies the validity of Sarnak's conjecture, see \cite{Ab-Ku-Le-Ru,Sa, Ta} for more details. Due to a recent result of Tao \cite{Ta2}, we know that Sarnak's conjecture on its turn implies the logarithmic version of the Chowla conjecture.
In particular, Sarnak's conjecture implies that all admissible blocks do appear on $\mob$,\footnote{This observation has been
communicated to us by W.\ Veech.} cf.\ comments on page~9 in~\cite{Sa}.

\paragraph{Measure-theoretic viewpoint}
While often one focuses on proving the M\"obius disjointness for a particular class of zero entropy dynamical systems (see the bibliography in \cite{Ta2}), our approach is more abstract and concentrates on the measure-theoretic aspects. The starting point for us is the following:
\begin{Question}
Let $(X,T)$ be a topological dynamical system and suppose that $x\in X$ is a generic point for some measure $\kappa$. Can measure-theoretic properties of $(X,\cb(X),\mu,T)$ imply the validity of~\eqref{momoe1}?
\end{Question}
In particular, we ask:
\begin{Question}\label{q2}
Do some measure-theoretic properties of $(Z,\cd,\kappa,R)$ imply M\"obius disjointness of all its uniquely ergodic models?
\end{Question}
A positive answer to Question~\ref{q2} was provided in \cite{Ab-Le-Ru1}, where this line of research was initiated. The key notion there is so-called AOP property (see~\eqref{momoe4} below) which  yields M\"obius disjointness and is an isomorphism invariant. In particular, we have M\"obius disjointness in all uniquely ergodic models of:
\begin{itemize}
\item
(zero entropy) affine automorphisms on compact connected Abelian groups, \cite{Ab-Le-Ru1}, in particular, totally ergodic discrete spectrum automorphisms;
\item
(more generally) all uniquely ergodic models of unipotent diffeomorphisms on nilmanifolds, \cite{Fl-Fr-Ku-Le}.
\end{itemize}
In an unpublished (earlier) version of the present article it was shown how to use the recent remarkable results on the average behavior of non-pretensious multiplicative functions on, so-called, short intervals \cite{Ma-Ra,Ma-Ra-Ta}, to see that all uniquely ergodic models of
\begin{itemize}
\item
finite rotations
\end{itemize}
are M\"obius disjoint. Moreover, the M\"obius disjointness of all uniquely ergodic models of an
\begin{itemize}
\item
arbitrary discrete spectrum automorphism
\end{itemize}
was proved in~\cite{Hu-Wa-Zh}. On the other hand, it is still unknown whether all uniquely ergodic models of horocycle flows are M\"obius disjoint.

Following the above lines, the following natural question emerges:
\begin{Question}\label{pytanie}
Does M\"obius disjointness in a certain uniquely ergodic model of an ergodic system yield M\"obius disjointness in all its uniquely ergodic models?
\end{Question}
Clearly, in the zero entropy case, the potential positive answer
to this question is supported by Sarnak's conjecture. Our main result -- \ref{thmAPm} -- yields an ``almost'' positive answer: the existence of one ``good'' model of a given system implies M\"obius disjointness of all its uniquely ergodic models.

\paragraph{Key properties}

Before providing more details, we need some preparation. Let $\bfu\colon\N\to\C$ be an arbitrary arithmetic function. In the applications, we will often take $\bfu=\mob$ which justifies the terminology, but the results are valid in general and hence we formulate them in this more abstract setting.



\begin{Def}[Sarnak property]\em
  A point $x\in X$ satisfies the \emph{Sarnak property} [relatively to $\bfu$] if, for any $f\in C(X)$,
  \begin{equation}
    \label{eq:defSarnak}
    \frac{1}{N} \sum_{n<N} f(T^nx)\bfu(n) \tend{N}{\infty} 0.
  \end{equation}
  We say that $(X,T)$ satisfies the Sarnak property [relatively to $\bfu$] if any point $x\in X$ satisfies the Sarnak property [relatively to $\bfu$]. Sarnak property of $(X,T)$ relatively to $\mob$ (or $\lio$)\footnote{$\lio$ stands for the Liouville function, that is, $\lio(n)=(-1)^{\Omega(n)}$, where $\Omega(n)$ denotes the number of prime divisors of $n$ counted with multiplicities. Note that $\lio(n)=\mob(n)$ for all $n$ which are square-free.}  is called M\"obius (or Liouville) disjointness.
\end{Def}


\begin{Def}[MOMO property: M\"obius Orthogonality on Moving Orbits, implicit in~\cite{Ab-Le-Ru1}]
\label{def:MOMO}
\em
We say that $(X,T)$ satisfies the \emph{MOMO property} [relatively to $\bfu$] if, for any increasing sequence of integers
$0=b_0<b_1<b_2<\cdots$ with $b_{k+1}-b_k\to\infty$, for any sequence $(x_k)$ of points in $X$, and any $f\in C(X)$,
\begin{equation}
    \label{eq:defMOMO}
    \frac{1}{b_{K}} \sum_{k< K} \sum_{b_k\le n<b_{k+1}} f(T^{n-b_k}x_k) \bfu(n) \tend{K}{\infty} 0.
  \end{equation}
\end{Def}
Note that the MOMO property is similar to Sarnak property, but the orbit can be changed from time to time, less and less often.

\begin{Def}[strong MOMO property]\label{def:SM}
\label{def:strongMOMO}\em
We say that $(X,T)$ satisfies the \emph{strong MOMO property} [relatively to $\bfu$] if, for any increasing sequence of integers
$0=b_0<b_1<b_2<\cdots$ with $b_{k+1}-b_k\to\infty$, for any sequence $(x_k)$ of points in $X$, and any $f\in C(X)$,
  \begin{equation}
    \label{eq:defMOMOSI}
    \frac{1}{b_{K}} \sum_{k< K} \left|\sum_{b_k\le n<b_{k+1}} f(T^{n-b_k}x_k) \bfu(n)\right| \tend{K}{\infty} 0.
  \end{equation}
\end{Def}
It follows directly from the definition that the strong MOMO property implies uniform convergence in~\eqref{eq:defSarnak}.\footnote{\label{stopa}Notice that if for a system $(X,T)$ and $f\in C(X)$ we do not have uniform convergence of the sums $\frac1N\sum_{n\leq N}f(T^n\cdot)\bfu(n)$, then there exists $\vep_0>0$ such that for each $k\geq1$, we can find $b_k\geq k$ and $x_k\in X$ for which $\left|\frac1{b_k}\sum_{n\leq b_k}f(T^nx_k)\bfu(n)\right|\geq\vep_0$. We can assume that the distances $b_k-b_{k-1}$ grow very rapidly so that we obtain
$\frac{1}{b_K}\sum_{k<K}(b_k-b_{k-1})\left|
\frac1{b_k-b_{k-1}}\sum_{b_{k-1}\leq n<b_k}f(T^nx_k)\bfu(n)\right|
\geq\sum_{k<K}\frac{b_k-b_{k-1}}{b_K}\vep_0/2\geq\vep_0/3,$
whence the strong MOMO property (relative to $\bfu$) fails.} Moreover, by taking $f=1$, it follows that the strong MOMO property implies the following:
\beq\label{eq:Mobius-like}
 \frac{1}{b_{K}} \sum_{k< K} \left|\sum_{b_k\le n<b_{k+1}} \bfu(n)\right| \tend{K}{\infty} 0\eeq
 for every sequence
$0=b_0<b_1<b_2<\cdots$ with $b_{k+1}-b_k\to\infty$.\footnote{We recall that by a result of Matomaki and Radziwi\l \l \ \cite{Ma-Ra} it follows that $\mob$ satisfies~\eqref{eq:Mobius-like}.}
In particular, $\frac1N\sum_{n<N}\bfu(n)\tend{N}{\infty}0$.

Note that, if we additionally assume that $\bfu$ is {\em bounded}, then the Sarnak property, the MOMO property 
and the strong MOMO property 
remain unchanged if we modify $\bfu$ on a subset of $\N$ of density 0. Note also that if $\bfu$ is bounded, by unique ergodicity,
to verify the MOMO or the strong MOMO property we only need to check the relevant convergence for a linearly $L^1$-dense set of continuous functions.

Finally, we clearly have
$$
\text{strong MOMO property} \Rightarrow \text{MOMO property}\Rightarrow \text{Sarnak property}
$$
(cf.\ Corollary~\ref{cor:12} below).

\paragraph{Main result}
By $M(X,T)$ we denote the set of $T$-invariant Borel probability measures, and $M^e(X,T)$ its subset of ergodic measures. If $\mu_1,\ldots,\mu_t\in M(X,T)$  then by ${\rm conv}(\mu_1,\ldots,\mu_t)$, we denote the corresponding  convex envelope. Recall that $\text{Q-gen}(x)$ is the set of measures in $M(X,T)$ for which $x\in X$ is quasi-generic.

We consider now an ergodic measure-theoretic dynamical system $(Z,\cd, \kappa,R)$, and we introduce three more properties (still relatively to $\bfu$), involving the above.

\setpropertytag{P1}
\begin{property}\label{newP1}
  There exists a topological dynamical system $(Y,S)$, and an $S$-invariant probability measure $\nu$ on $Y$, such that
  \begin{itemize}
    \item $(Y,S)$ satisfies the strong MOMO property,
    \item $(Y,\cb(Y),\nu,S)$ is measure-theoretically isomorphic to $(Z,\cd,\kappa,R)$.
  \end{itemize}
\end{property}

\setpropertytag{P2}
\begin{property}\label{newP2}
  For any topological dynamical system $(X,T)$ and any $x\in X$, if there exists a finite number of $T$-invariant measures $\mu_j$, $1\le j\le t$, such that
  \begin{itemize}
    \item for each $j$, $(X,\cb(X),\mu_j,T)$ is measure-theoretically isomorphic to $(Z,\cd,\kappa,R)$,
    \item  any measure for which $x$ is quasi-generic is a convex combination of the measures $\mu_j$, i.e.\ $\text{Q-gen}(x)\subset{\rm conv}(\mu_1,\ldots,\mu_t)$,
  \end{itemize}
  then
$x$ satisfies the Sarnak property. \label{p2}
\end{property}


\setpropertytag{P3}
\begin{property}\label{P3}
  Any uniquely ergodic model $(Y,S)$ of $(Z,\cd,\kappa,R)$ satisfies the strong MOMO property.
\end{property}

\begin{named}{Main Theorem}{}
\label{thmAPm}
  Properties~\ref{newP1}, \ref{newP2} and \ref{P3} are equivalent.
\end{named}
We observe that, since there always exists a uniquely ergodic model of $(Z,\cd,\kappa,R)$ by the Jewett-Krieger theorem, the implication \ref{P3} $\Longrightarrow$ \ref{newP1} is obvious. The two other implications are treated in separate subsections.
While the proof of the implication \ref{newP2}$\Rightarrow$\ref{P3} uses some ideas from \cite{Ab-Le-Ru1}, the proof of the implication \ref{newP1}$\Rightarrow$\ref{newP2} heavily depends on the main ideas from \cite{Hu-Wa-Zh}.

\paragraph{AOP and Property~\ref{P3}}
Recall  that an ergodic automorphism $R$ is said to have AOP (asymptotically orthogonal powers) \cite{Ab-Le-Ru1} if for each $f,g\in L^2_0(Z,\cd,\kappa)$, we have
\beq\label{momoe4}
\lim_{\mathscr{P}\ni r,s\to\infty, r\neq s} \sup_{\kappa\in J^e(R^r,R^s)}\left|\int_{X\times X} f\ot g\,d\kappa\right|=0.\footnote{$\mathscr{P}$ stands for the set of prime numbers. The set $J^e(R^r,R^s)$ consists of $R^r\times R^s$-invariant measures on $Z\times Z$ which are ergodic and whose both projections on $Z$ are $\kappa$.}
\eeq
The AOP property implies zero entropy.
It also implies total ergodicity. Indeed, as clearly AOP is closed under taking factors, we merely need to notice that $Rx=x+1$ acting on $\Z/k\Z$ with $k\geq2$ has no AOP property. The latter easily follows from the Dirichlet theorem on primes in arithmetic progressions.

Clearly, if the powers of $R$ are pairwise disjoint\footnote{This is a ``typical'' property of an automorphism of a probability standard Borel space~\cite{Ju}.} in the Furstenberg sense \cite{Fu} then $R$ enjoys the AOP property. The
AOP property  of $(Z,\cd,\kappa,R)$ implies the MOMO property in every uniquely ergodic model of $R$ \cite{Ab-Le-Ru1} relatively to a multiplicative\footnote{ Multiplicativity means that $\bfu(1)=1$ and
$\bfu(mn)=\bfu(m)\bfu(n)$ whenever $m$ and $n$ are relatively prime.} $\bfu\colon\N\to\C$, $|\bfu|\leq1$, satisfying $\frac1N\sum_{n\leq N}\bfu(n)\to 0$ when $N\to\infty$. In particular, the MOMO property relatively to $\mob$ (or $\lio$) holds.
We will show that AOP implies Property~\ref{P3} (see Section~\ref{sec:aop2P1}), which results in the following.

\begin{Th}\label{thmB}
Let $\bfu\colon\N\to\C$ be  multiplicative, $|\bfu|\leq1$. Suppose that $(Z,\cd,\kappa,R)$ satisfies AOP. Then the following are equivalent:
\begin{itemize}
\item
$\bfu$ satisfies \eqref{eq:Mobius-like};
\item
The strong MOMO property relatively to $\bfu$ is satisfied in each uniquely ergodic model $(X,T)$ of $R$.
\end{itemize}
In particular, if the above holds, for each $f\in C(X)$, we have
\[ \frac1N\sum_{n\leq N}f(T^nx)\bfu(n)\to 0\] uniformly in $x\in X$.
\end{Th}

\begin{Cor}\label{corC} Assume that $(Z,\cd,\kappa,R)$ enjoys the AOP property. Then in each uniquely ergodic model $(X,T)$ of $R$, we have
\beq\label{momoe5}
\frac1M\sum_{M\leq m<2M}\left|\frac1H\sum_{m \leq h<m+H}f(T^hx)\mob(h)\right|\to0\text{ when }H\to\infty,H/M\to0\eeq
for all $f\in C(X)$, $x\in X$.\end{Cor}

Note that even if we consider $(Z,\cd,\kappa,R)$ with the property that all of its powers are disjoint, the validity of~\eqref{momoe5} in each uniquely ergodic model $(X,T)$ of $R$ is a new result.\footnote{M\"obius disjointness for this case is already noticed in \cite{Bo-Sa-Zi}.}

\paragraph{Consequences: zero entropy}

The proof of the \ref{thmAPm} implies that we have:
\begin{Cor}\label{cor:12}
The following are equivalent:
\begin{itemize}
\item
Sarnak's property holds for all systems of zero topological entropy,
\item
strong MOMO property holds for all systems of zero topological entropy.
\end{itemize}
In particular, Sarnak's conjecture holds if and only if the strong MOMO property [relatively to $\mob$] holds for all systems of zero topological entropy.
\end{Cor}

By Corollary~\ref{cor:12} and by the first comment after Definition \ref{def:SM}, we obtain immediately:
\begin{Cor}\label{p:unif}
If Sarnak's conjecture is true then for all zero entropy systems $(X,T)$ and $f\in C(X)$, we have
$\frac1N\sum_{n\leq N}f(T^nx)\mob(n)\to0$
uniformly in $x\in X$.
\end{Cor}

\begin{Remark}\label{rk:more}
One can think of the assertion of Corollary~\ref{p:unif} (with $\mob$ replaced by a sequence $\bfu$) as of a statement on ``unique ergodicity’’. For example, if $(Y,S)$ is a topological system and $y\in Y$ is a generic point for a Bernoulli measure $\nu$ and if $g\in C(Y)$ has $\nu$-mean zero then for every zero topological  entropy system $(X,T)$ and every $f\in C(X)$ the averages $\frac1N\sum_{n\leq N} f(T^nx)g(S^ny)\to 0$ uniformly in $x\in X$ (here $\bfu(n)=g(S^ny))$. However, it is easy to prove uniform convergence here by repeating any of the classical proofs of the fact that in a uniquely ergodic system ergodic averages converge uniformly.\footnote{If $|\frac1{b_k}\sum_{n\leq b_k}f(T^nx_k)g(S^ny)|\geq \vep_0$ (cf.\ footnote~\ref{stopa}) then by considering $\frac1{b_k}\sum_{n\leq b_k}\delta_{(T^nx_k,S^n\bfu)}$, we can assume that it converges to a joining of a measure with zero entropy and $\nu$ (which is Bernoulli). Hence it is the product measure, and we easily obtain a contradiction with the fact
that $\int g\, d\nu=0)$.}
\end{Remark}

\paragraph{Consequences: positive entropy}
Let $\D$ stand for the unit disc, $D_L:=L\D$ for $L>0$ and $S$ for the shift in $\D_L^{\Z}$. Given $\bfu\in(\D_L)^{\Z}$, we denote by $X_{\bfu}$ the subshift generated by $\bfu$.

\begin{Cor}\label{p1j}
Fix $\kappa\in M^e((\D_L)^{\Z},S)$. Let $(X,T)$ be any uniquely ergodic model of $((\D_L)^{\Z},\kappa,S)$. Let  $\bfu\in(\D_L)^{\Z}$ be such that $\text{Q-gen}(\bfu)\subset{\rm conv}(\kappa_1,\ldots,\kappa_m)$, where $((\D_L)^{\Z},\kappa_j,S)$ for $j=1,\ldots,m$  is measure-theoretically isomorphic to $((\D_L)^{\Z},\kappa,S)$.  Assume that $\bfv\in(\D_L)^{\Z}$ and $\frac1N\sum_{n\leq N}\bfu(n)\ov{\bfv(n)}$ does not converge to zero, i.e., $\bfu$ and $\bfv$ correlate. Then the system $(X,T)$ does not satisfy the strong MOMO property (relatively to $\bfv$).
\end{Cor}


In particular, we can use Corollary~\ref{p1j} for $\bfv=\bfu$, assuming that $\kappa\neq\delta_{(\ldots0.00\ldots)}$.
Hence, see Section~\ref{s:momo} for details, if $(X,T)$ is fixed then all  points $\bfu$ as above are ``visible'' in $X$ in the following sense:
\begin{equation}\label{m1}
\parbox{.85\textwidth}{$\left(\exists \vep_0>0\right)\;\left(\exists f\in C(X)\right)\;\left(\exists (x_k)\subset X\right)$\\
$\left(\exists A=\bigcup_{k=1}^\infty[a_k,c_k)\subset\N\;\text{of disjoint intervals}, c_k-a_k\to\infty,\;\overline{d}(A)>0\right)$ such that
$\frac1{c_k-a_k}\left|\sum_{n=0}^{c_k-a_k-1}f(T^nx_k)\bfu(a_k+n)\right|\geq\vep_0$ for each $k\geq1$.}
\end{equation}

Recently, Downarowicz and Serafin \cite{Do-Se} constructed  positive entropy homeomorphisms of arbitrarily large entropy  which are M\"obius (or Liouville) disjoint. The following natural question arises:
\begin{Question}
Does there exist an ergodic positive entropy measure-theoretic system whose all uniquely ergodic models are M\"obius (or Liouville) disjoint?
\end{Question}
\noindent
A partial (negative) answer is given by the following two results:

\begin{Cor}\label{c11}
Assume that $\bfu\in (\D_L)^{\Z}$ is generic for a Bernoulli measure $\kappa$. Let $\bfv\in(\D_L)^\Z$, $\bfu$ and $\bfv$ correlate. Then for each dynamical system $(X,T)$ with $h(X,T)>h((\D_L)^{\Z},\kappa,S)$, we do not have the strong MOMO property relatively to $\bfv$.\end{Cor}

\begin{Cor}\label{dodENT}
Assume that the Chowla conjecture holds for $\lio$.\footnote{The Chowla conjecture is equivalent to saying that $\lio$ is a generic point for the Bernoulli measure $B(1/2,1/2)$ for the shift on $\{-1,1\}^{\N}$, cf.\ footnote~\ref{f1}.}  Then
no topological system $(X,T)$ with positive entropy satisfies the strong MOMO property relatively  to $\lio$.
\end{Cor}
 In fact, by the proof of the implication \ref{newP2}$\:\Rightarrow\:$\ref{P3}, it follows that whenever an ergodic  measure-theoretic system $(Z,{\cal D},\kappa,R)$ has positive entropy then  there must exist a topological system which has at most three ergodic measures, all yielding systems isomorphic to $R$, and for which the Sarnak property (relatively to $\lio$) does not hold. Corollary~\ref{cor:12}, Corollary~\ref{p:unif} and Corollary~\ref{dodENT} seem to yield a much better understanding of Sarnak's conjecture (at least for the Liouville function). It would be an interesting challenge to construct a uniquely ergodic model of the Bernoulli system $B(1/2,1/2)$ which has no strong MOMO property relatively to $\lio$.

\paragraph{Property~\ref{p2} and examples}
In the second part of the paper, we will concentrate on examples of automorphisms enjoying Property~\ref{p2} relatively to any bounded multiplicative function\emph{} $\bfu$ satisfying~\eqref{eq:Mobius-like}.  We will show a general criterion to lift the strong MOMO property to extensions and deal with coboundary extensions of homeomorphisms satisfying the strong MOMO property.
The validity of~\ref{p2} in a large subclass of so-called generalized Morse sequences \cite{Ke} then follows by exploiting the idea of lifting the strong MOMO property by some cocycle extensions (see Theorem~\ref{t:SMext} below). We will show in particular that Property \ref{p2} (hence also the strong MOMO property holds) holds for all uniquely ergodic models of: all unipotent diffeomorphisms on nilmanifolds, all transformations with discrete spectrum, typical automorphism of a probability standard Borel space, systems coming from bijective substitutions and some other ``close'' to that; in particular, the classical Thue-Morse and Rudin-Shapiro systems and, finally,  for systems determined by so-called Kakutani sequences. For the M\"obius disjointness  in the classes of systems listed above, see \cite{Bo,Bo-Sa-Zi,Da,Fl-Fr-Ku-Le,Gr,Mu,Ve}.

\paragraph{Strong MOMO property and examples}
We will now show what the strong MOMO property relatively to the Liouville function can mean in the case of Kakutani sequences.

  Using the sequence $1,2,2^2,\ldots$, each natural number $n\geq1$ can be written uniquely as $n=\sum_{j\geq0}\vep_j2^j$ with $\vep_j=0$ or~$1$. Then, we can consider
the sequence $s_2(n):=\sum_{j\geq 0}\vep_j$ mod~2, $n\geq1$. Using the sequence $2=p_1<p_2<p_3<\ldots$ of consecutive primes numbers, each natural number $n\geq2$ can be written uniquely as $n=\prod_{j\geq1}p_j^{\alpha_j}$ with $\alpha_j\geq0$, $j\geq1$. Then, we can consider the sequence $b(n):=\sum_{j\geq1}\alpha_j$ mod~2, $n\geq1$. Properties of $(s_2(n))$ and $(b(n))$ concern the additive and multiplicative structure of $\N$, respectively. Hence, we can expect some form of independence of the two sequences. Indeed, we will show that $(-1)^{s_2(n)}$ and $(-1)^{b(n)}$, $n\geq1$,  are uncorrelated in a strong way. Note that $(-1)^{s_2(n)}=(-1)^{x(n)}$, where $x\in\{0,1\}^{\N}$ is the classical Thue-Morse sequence, while $\lio(n):=(-1)^{b(n)}$, where $\lio$ is the classical Liouville function.\footnote{As $n\mapsto b(n)$ is completely additive, the Liouville function is a completely multiplicative function.} In fact, our result will be more general. Following B.\ Green \cite{Gr},  let $A\subset \N$. Consider
$x=x_A\in\{0,1\}^{\N}$ such that
$$
x(n)=\sum_{i\in A} \vep_i\;\mbox{mod}\;2,\;\mbox{where}\; n=\sum_{i\geq 0}\vep_i 2^i.$$
As explained in \cite{Ab-Ka-Le}, each $x=x_A$ is a Kakutani sequence (and each Kakutani sequence determines an $A$).\footnote{This observation is due to C.\ Mauduit.}

It follows from the strong MOMO property (with respect to $\lio$) of the system determined by $x_A$ that
$$
\frac1{b_{K}}\sum_{k<K}\left|\sum_{b_k\leq n<b_{k+1}}(-1)^{x_A(n)}\lio(n)\right|\tend{K}{\infty} 0.$$
Now, $\frac1N\sum_{n\leq N}(-1)^{x_A(n)}\to0$\footnote{This follows from the fact that for the unique invariant measure $\mu_{x_A}$ for the subshift determined by $x_A$, we have $\int (-1)^{z(0)}\,d\mu_{x_A}(z)=0$.} and $\frac1N\sum_{n\leq N}\lio(n)\to0$ (the latter is equivalent to the PNT).  More than that, the same property holds on each short interval for the first sequence (by the unique ergodicity of the system determined by a Kakutani sequence) or on a typical short interval for the Liouville function by a result of \cite{Ma-Ra}.
Recall that if we have two random variables $X,Y$ taking two values $\pm1$ with probability 1/2 then they are independent if and only if $\int XY=0$.
By all this, we have proved the following form of independence of the sequences 
$(x_A(n))$ and $(b(n))$:

\begin{Prop}\label{propD} We have
$$\frac1M\sum_{M\leq m<2M}\left|\frac1H\sum_{m\leq h<m+H}(-1)^{x_A(h)}\lio(h)\right|\to0$$
when $H\to\infty$,  $H/M\to0$.
Moreover,  for each $e,f\in\{-1,1\}$,
\[
  \frac1M\sum_{M\leq m<2M} \left|\frac1H\left|\{m\leq h< m+H:\: (-1)^{x_A(h)}=e,\lio(h)=f\}\right|-\frac14\right|\to0
\]
when $H\to\infty$,  $H/M\to0$.
\end{Prop}

\section{MOMO property in models of an ergodic system}

\subsection{Proof of \ref{newP1} $\Longrightarrow$ \ref{newP2}}\label{sec:newP1_to_newP2}
This proof strongly relies on ideas borrowed from~\cite{Hu-Wa-Zh}. Let $(Y,S)$ and $\nu$ be given by \ref{newP1}, and let $(X,T)$, $\mu_1,\ldots,\mu_t$ and $x$ be as in the assumptions of \ref{newP2}. In particular, for each $1\le j\le t$,  $(X,\cb(X),\mu_j,T)$ is measure-theoretically isomorphic to $(Y,\cb(Y),\nu,S)$.
We also fix a continuous function $f$ on $X$, and $0<\varepsilon<\frac12$.

Since the measures $\mu_j$ are ergodic for $T$, we can find $T$-invariant disjoint Borel subsets $X_j$, $1\le j\le t$ with $\mu_j(X_j)=1$, and measure-theoretic isomorphisms $\phi_j\colon (X_j,\cb(X_j),\mu_j,T)\to (Y,\cb(Y),\nu,S)$.
For each $1\le j\le t$, Lusin's theorem now provides a compact subset $W_j\subset X_j$, with $\mu_j(W_j)>1-\varepsilon^4$, such that the restriction $\phi_j|_{W_j}$ is continuous. Then this restriction is in fact a homeomorphism between $W_j$ and its image $\phi_j(W_j)$, which is a compact subset of $Y$.
The function $f\circ \phi_j^{-1}$ is continuous on this compact subset, so by the Tietze extension theorem it can be extended to a continuous function $g_j$ on the entire space $Y$, with $\|g_j\|_\infty=\|f\|_\infty$. The following observation will be useful: for any $w\in X$ and any $s\ge0$,
\begin{equation}
  \label{eq:window}
  [w\in W_j \text{ and }T^sw\in W_j] \Longrightarrow f(T^sw)=g_j\left(\phi_j(T^sw)\right) = g_j\left(S^s(\phi_jw)\right).
\end{equation}
Informally, the preceding observation means that these compact sets $W_j$ can be used as ``windows" through which we can see the behaviour of the dynamical system $(Y,S)$. We would like to see this behaviour along long pieces of orbits, but these windows are not $T$-invariant. This is why    we need to define, for each $1\le j\le t$ and each integer $L\ge1$,  the following subset of $W_j$:
\[
  B_j(L):=\left\{w\in W_j: \frac{1}{L}\sum_{\ell<L} \ind{W_j}(T^\ell w) > 1-\varepsilon^2\right\}.
\]
Observe that $B_j(L)$ can be written as the finite union of all sets of the form $T^{-\ell_1}W_j\cap\cdots\cap T^{-\ell_r}W_j$, where $\{\ell_1,\ldots,\ell_r\}$ ranges over all finite subsets of $\{0,\ldots,L-1\}$ of cardinality larger than $1-\varepsilon$. Each $T^{-\ell}W_j$ being compact, $B_j(L)$ is compact.
On $W_j\setminus B_j(L)$, we have the inequality
\[
 \frac{1}{L}\sum_{\ell<L} \ind{X\setminus W_j}(T^\ell w) \ge \varepsilon^2.
\]
Therefore, we have
\[
   \varepsilon^2\mu_j\bigl(W_j\setminus B_j(L)\bigr)\le\int_X  \frac{1}{L}\sum_{ \ell<L} \ind{X\setminus W_j}\circ T^\ell \, d\mu_j = \mu_j(X\setminus W_j) <\varepsilon^4,
\]
and this yields $\mu_j(B_j(L))>1-\varepsilon$.

Let $d(\cdot,\cdot)$ be the distance on $X$. To each integer $L\ge1$, we also associate a positive number $\eta(L)$, small enough so that for any $w,w'\in X$,
\begin{equation}
  \label{eq:equicontinuity}
  d(w,w')<\eta(L)\Longrightarrow \forall 0\le n<L, |f(T^nw)-f(T^nw')| < \varepsilon.
\end{equation}

The following lemma is the only place where we use the assumption on the quasi-genericity of the point $x$.
\begin{Lemma}
  \label{lemma:density}
  For each $L\ge1$, let $B(L)$ be the disjoint union $B_1(L)\sqcup\cdots\sqcup B_t(L)$, which is also a compact subset of $X$.
  Then
  \[
  \limsup_{N\to\infty}\frac{1}{N}
  \#\Bigl\{ n\in\{0,\ldots,N-1\}: d\bigl(T^nx,B(L)\bigr)\ge \eta(L)\Bigr\} < \varepsilon.\]
\end{Lemma}
\begin{proof}
  Let $(N_i)$ be an increasing sequence of positive integers along which the convergence to the limit superior in the statement of the lemma holds. Extracting if necessary a subsequence, we can assume that $x$ is quasi-generic, along this sequence $(N_i)$, for a $T$-invariant measure $\mu$ which is of the form $\mu=\alpha_1\mu_1+\cdots+\alpha_t\mu_t$, with $\alpha_j\ge0$ and $\alpha_1+\cdots+\alpha_t=1$. Since $\mu_j(B(L))\ge \mu_j(B_j(L))>1-\varepsilon$ for each $j$, we also have $\mu(B(L))>1-\varepsilon$.

  Using again the Tietze extension theorem, we can construct a continuous function $h\in C(X)$, $0\le h\le 1$, satisfying $h(w)=0$ if $w\in B(L)$, and $h(w)=1$ if $d\bigl(w,B(L)\bigr)\ge \eta(L)$. Then, we have
  \begin{align*}
    & \limsup_{N\to\infty}\frac{1}{N}
  \#\Bigl\{ n\in\{0,\ldots,N-1\}: d\bigl(T^nx,B(L)\bigr)\ge \eta(L)\Bigr\} \\
  = & \lim_{i\to\infty}\frac{1}{N_i}
  \#\Bigl\{ n\in\{0,\ldots,N_i-1\}: d\bigl(T^nx,B(L)\bigr)\ge \eta(L)\Bigr\} \\
  \le & \lim_{i\to\infty} \frac{1}{N_i} \sum_{n<N_i} h(T^nx) \\
  = &\int_X h\,d\mu
  \le  \mu(X\setminus B(L)) < \varepsilon.
  \end{align*}
\end{proof}

We now fix an increasing sequence of integers $1\le L_1<L_2<\cdots$.
We also choose an increasing sequence of integers $(M_i)_{i\geq1}$ such that
\[
\lim_{i\to \infty} \frac{1}{M_i} \left|\sum_{ n<M_i} f(T^nx)\bfu(n)\right| = \limsup_{N\to\infty}\frac{1}{N} \left|\sum_{ n<N} f(T^nx)\bfu(n)\right|.
\]
Set $M_0:=0$. With the help of Lemma~\ref{lemma:density}, we can assume, passing to a subsequence if necessary, that for each $i\ge1$,
\begin{equation}
  \label{eq:defMm}
  \frac{1}{M_{i}-M_{i-1}} \#\Bigl\{ n\in\{M_{i-1},\ldots,M_i-1\}: d\bigl(T^nx,B(L_i)\bigr)\ge \eta(L_i)\Bigr\} < \varepsilon.
\end{equation}
We can also assume that for each $i\ge1$, $M_i$ is large enough to have
\begin{equation}
  \label{eq:LiNi}
  L_i<\varepsilon M_{i}.
\end{equation}

Then, for any integer $b\ge0$, there exists a unique $i\ge1$ such that $M_{i-1}\le b< M_i$. We  say that $b$ is \emph{good} if $d\bigl(T^bx,B(L_i)\bigr)<\eta(L_i)$. Observe that, by~\eqref{eq:defMm}, the proportion of good integers between $M_{i-1}$ and $M_i$ is always at least $1-\varepsilon$.

Finally, we inductively define a third increasing sequence of integers $0=b_0<b_1<b_2<\cdots$ in the following way. Let $b_1$ be the smallest good integer $b\ge1$. Assume that  the integer $b_k$ has already been defined ($k\ge1$), and that it is a good integer. Let $i\ge1$ be such that $M_{i-1}\le b_k<M_i$. Since $b_k$ is good, there exist $1\le j_k\le t$ and a point $x_k\in
B_{j_k}(L_i)$ such that $d(T^{b_k}x,x_k) < \eta (L_i)$. By the definition of $\eta(L_i)$, this implies that for any $0\le s<L_i$,
$|f(T^{b_k+s}x) - f(T^sx_k)|<\varepsilon$.  But by the definition of $B_{j_k}(L_i)$, the number of integers $s$, $0\le s<L_i-1$, such that
$T^sx_k\in W_{j_k}$ is at least $(1-\varepsilon)L_i$. Moreover, since $x_k\in W_{j_k}$, by~\eqref{eq:window}, for each such $s$, we have
\[
  f(T^sx_k) = g_{j_k} (S^s y_k), \text{ where }y_k:=\phi_{j_k}(x_k),
\]
which yields $|f(T^{b_k+s}x) - g_{j_k} (S^s y_k)| < \varepsilon$.
We therefore get
\begin{equation*}
  \sum_{s<L_i} |f(T^{b_k+s}x) - g_{j_k} (S^s y_k)| < \varepsilon L_i + 2\varepsilon L_i \|f\|_\infty.
\end{equation*}
Now, we define $b_{k+1}$ as the smallest integer $b\ge b_k+L_i$ which is good. The above inequality then gives
\begin{multline}
  \label{eq:between}
  \sum_{b_k\le n<b_{k+1}} |f(T^{n}x) - g_{j_k} (S^{n-b_k} y_k)|
  < \varepsilon (b_{k+1}-b_k)(1 + 2 \|f\|_\infty) \\
  +2\|f\|_\infty \#\{b_k\le n<b_{k+1}:\ n\text{ is not good}\}.
\end{multline}
And since $L_i\to\infty$, we also have $b_{k+1}-b_k\to\infty$.

Now, fix $i\ge1$. Let $K\ge1$ be the largest integer such that $b_{K}\le M_i$. We want to approximate
the sum
\[ S_{M_i}:=\sum_{ n<M_i} f(T^nx) \bfu(n)\]
by the following expression coming from the dynamical system $(Y,S)$
\[ E_K:=\sum_{ k<K} \sum_{b_k\le n<b_{k+1}} g_{j_k}  (S^{n-b_k} y_k) \bfu(n).\]
Considering that~\eqref{eq:between} holds for $1\le k< K$, observing that all integers $n\in\{0=b_0,\ldots,b_1-1\}$
are not good, and that the number of integers less than $M_i$ which are not good is bounded by $\varepsilon M_i$, we get
\[
  \left| S_{M_i} - E_K \right| < \varepsilon M_i (1+4\|f\|_\infty) + 2\|f\|_\infty (M_i-b_{K}).
\]
Now, by the definition of $K$, $M_i-b_{K}$ can be bounded by $L_{i}+\#\{n<M_i:n\text{ is not good}\}$, which is itself bounded by
$2\varepsilon M_i$ by~\eqref{eq:LiNi}.
Therefore, there exists a constant $C>0$ such that
\begin{multline*}
   \limsup_{N\to\infty}\frac{1}{N} \left|\sum_{n<N} f(T^nx)\bfu(n)\right|
   =\lim_{i\to \infty} \frac{1}{M_i} \left|S_{M_i}\right|
   \le \limsup_{K\to\infty} \frac{1}{b_{K}}|E_K| + C\varepsilon.
\end{multline*}
It only remains to do the following estimation:
\begin{align*}
  \frac{1}{b_{K}}|E_K| &\le
  \frac{1}{b_{K}} \sum_{ k<K} \left| \sum_{b_k\le n<b_{k+1}} g_{j_k}  (S^{n-b_k} y_k) \bfu(n) \right| \\
  & \le \sum_{1\le j\le t}
  \frac{1}{b_{K}} \sum_{ k<K} \left| \sum_{b_k\le n<b_{k+1}} g_{j}  (S^{n-b_k} y_k) \bfu(n) \right|,
\end{align*}
which goes to 0 as $K\to\infty$ by the strong MOMO property of $(Y,S)$. This concludes the proof of the implication
\ref{newP1} $\Longrightarrow$ \ref{newP2}.

\subsection{Proof of \ref{newP2} $\Longrightarrow$ \ref{P3}}
Before we begin the proof, let us recall the following elementary result.

\begin{Lemma}\label{lemma:cone}
Assume that $(c_n)\subset\C$ and $(m_n)\subset\N$. Then:
\begin{enumerate}[label=({\Alph*})]
\item
If for each sequence $(\vep_k)\subset\{-1,1\}^{\N}$, we have
$$\frac1{m_N}\sum_{n\leq N}\vep_nc_n\tend{N}{\infty} 0\text{ then }\frac1{m_N}\sum_{n\leq N}|c_n|\tend{N}{\infty} 0;$$
\item
If the sequence $(c_n)$ is contained in a convex cone $C$, and which is not a half-plane, then
$\frac1{m_N}\sum_{n\leq N}c_n\tend{N}{\infty} 0$ if and only if $\frac1{m_N}\sum_{n\leq N}|c_n|\tend{N}{\infty} 0$.
\end{enumerate}
\end{Lemma}
\begin{proof}
To see (A) write ($c_n=a_n+ib_n$, $a_n,b_n\in\R$) $$\frac1{m_N}\sum_{n\leq N}\vep_nc_n=
\frac1{m_N}\sum_{n\leq N}\vep_na_n+i\frac1{m_N}\sum_{n\leq N}\vep_nb_n\tend{N}{\infty} 0$$ to reduce the problem to $c_n\in\R$, and \added{then} set $\vep_n:={\rm sign}(c_n)$.

To see (B), first multiply the whole sequence by some $e\in \C$, $|e|=1$ to assume without loss of generality that
$C$ is the cone \deleted{given}\added{delimited} by the \added{half-}lines $y=a_1x$, $y=a_2x$, $a_1\neq0\neq a_2$, and $x\geq0$. It follows that there is a constant $\gamma\geq1$ such that
for each $c\in C$, $|c|\leq\gamma {\rm Re}(c)$. \deleted{By assumption,}\added{If $\frac1{m_N}\sum_{n\leq N}c_n\tend{N}{\infty} 0$, then}  $\frac1{m_N}\sum_{n\leq N}{\rm Re}(c_n)\tend{N}{\infty} 0$, whence
$$ \frac1{m_N}\sum_{n\leq N}|c_n|\leq \gamma\frac1{m_N}\sum_{n\leq N}{\rm Re}(c_n)\tend{N}{\infty}0.$$
\end{proof}

Let $(Y,S)$ be a uniquely ergodic model of $(Z,\cd, \kappa,R)$, and let $\nu$ be the unique $S$-invariant measure.
We fix a continuous function $f\in C(Y)$, an increasing sequence of integers $0=b_0<b_1<b_2<\cdots$ with $b_{k+1}-b_k\to\infty$, and
a sequence of points $(y_k)_{k\ge0}$ in $Y$. Let $\A:=\{1,e^{i2\pi/3},e^{i4\pi/3}\}$ be the set of third roots of unity. For each $k$, let $e_k\in\A$ be such that
\[
  e_k \left(  \sum_{b_k\le n<b_{k+1}} f(S^{n-b_k}y_k) \bfu(n)  \right)
\]
belongs to the closed convex cone $\{0\}\cup\{z\in\C:\arg(z)\in[-\pi/3,\pi/3]\}$.
Then, by Lemma~\ref{lemma:cone}, the convergence that we want to prove:
\[
  \frac{1}{b_K} \sum_{k<K}  \left| \sum_{b_k\le n<b_{k+1}} f(S^{n-b_k}y_k) \bfu(n) \right|  \tend{K}{\infty} 0
\]
is equivalent to the convergence
\begin{equation}
  \label{eq:convergencetoprove}
  \frac{1}{b_K} \sum_{ k<K}  \sum_{b_k\le n<b_{k+1}}  e_k f(S^{n-b_k}y_k) \bfu(n)  \tend{K}{\infty} 0.
\end{equation}

Now, we introduce a new topological dynamical system: let $X:=(Y\times \A)^\N$, and let $T$ be the shift on $X$. We define a
particular point $x\in X$ by setting
\begin{equation}\label{xn}
x_n:=(S^{n-b_k}y_k,e_k)\text{ whenever }b_k\le n<b_{k+1}.\end{equation}
\added{Let $\mu$ be a measure for which $x$ is quasi-generic, along a sequence $(N_r)$.}
Since $b_{k+1}-b_k\to\infty$, we have
$$
\left(\frac1{N_r}\sum_{n<N_r}\delta_{T^nx}\right)\Bigl(\{v\in X:\: (v_1,a_1)=(Sv_0,a_0)\}\Bigr)\tend{N_r}{\infty} 1.$$
Since the set $\{v\in X:\: (v_1,a_1)=(Sv_0,a_0)\}$ is closed, \added{by Portmanteau theorem} it must be of full measure $\mu$\deleted{ for which $x$ is quasi-generic}  in $(X,T)$. Moreover, such a measure $\mu$ must be $T$-invariant, hence, it is concentrated on the set of sequences
which are of the form $\Bigl( (y,a), (Sy,a), (S^2y,a),\ldots  \Bigr)$ for some $y\in Y$ and some $a\in \A$. Since $\mu$ is $T$-invariant,  its marginal given by the first $y$-coordinate is $S$-invariant, so it is equal to $\nu$ by unique ergodicity. On the other hand, the marginal given by the $a$-coordinate must be of the form $\alpha_0 \delta_1 + \alpha_1 \delta_{e^{i2\pi/3}}+\alpha_2 \delta_{e^{i4\pi/3}}$. Using now the fact that any ergodic system is disjoint from the identity, $\mu$ must be the direct product of its marginals on $Y^\N$ and $\A^\N$.
Hence, $\mu$ must be of the form $\alpha_0\mu_0+\alpha_1\mu_1+\alpha_2\mu_2$, where $\mu_j$ is the pushforward of $\nu$ by the map
$y\mapsto \Bigl( (y,e^{i2\pi j/3}), (Sy,e^{i2\pi j/3}), (S^2y,e^{i2\pi j/3}),\ldots  \Bigr)$.
Then $(X,\cb(X),\mu_j,S)$ is measure-theoretically isomorphic to $(Y,\cb(Y),\nu,S)$, hence also to $(Z,\cd,\kappa,R)$: the assumptions needed to apply \ref{newP2} are fulfilled. It follows that $x$ satisfies the Sarnak property, and if we write the corresponding convergence for the continuous function $g\in C(X)$ defined by $g\Bigl( (w_0,a_0), (w_1,a_1), (w_2,a_2),\ldots  \Bigr):=a_0f(w_0)$, we exactly get~\eqref{eq:convergencetoprove}.\footnote{\label{f:stopka1}Note also that $\int_Xg\,d\mu_j=a_0\int_Xf(w_0)\,d\mu_j=a_0\int_Y f\,d\nu$.}

Note that, if in the above proof $(Z,\cd,\kappa,R)$ is an arbitrary system of zero entropy, and $(Y,S)$ is\deleted{its} any uniquely ergodic model \added{of $(Z,\cd,\kappa,R)$}, then the topological entropy of $S$ is zero. The same property holds for the topological system $(X,T)$ constructed above. If Sarnak's conjecture holds, the system $(X,T)$ is M\"obius disjoint and hence the strong MOMO property holds for $(Y,S)$. It follows that if Sarnak's conjecture holds, then \added{the strong MOMO property is satisfied} in each zero entropy uniquely ergodic system\deleted{the strong MOMO property is satisfied}. In fact, we have even the assertion of Corollary~\ref{cor:12} which are now ready to prove:
\begin{proof}[Proof of Corollary~\ref{cor:12}]
In view of the above proof of \ref{newP2} $\Longrightarrow$ \ref{P3}, what we need to show is that the orbit closure of $x\in X$ defined in~\eqref{xn} under $T$ has zero topological entropy. Suppose first that $x$ is quasi-generic for some measure $\mu$. Denote by $\mu^{(1)}$ the marginal of $\mu$ given by the first $(y,e)$-coordinate. Arguing as above, we obtain that $\mu^{(1)}$ is a measure invariant under $S\times I$, in particular, it has zero entropy. Moreover, $\mu$ is the image of $\mu^{(1)}$ by the map $(I\times I)\times(S\times I)\times (S^2\times I)\times \dots$, i.e.\ also has zero entropy.

Consider now $z$ in the orbit closure of $x$ and suppose that it is a quasi-generic point for some measure. If $n_j\to \infty$ and $z=\lim_{j\to\infty}T^{n_j}x$, then either
$$
z=((y,e),(Sy,e),(S^2y,e),\dots) \text{ for some }(y,e)\in Y\times \mathbb{A}
$$
or
\begin{multline*}
z=((y_1,e_1),(Sy_1,e_1),\dots,(S^\ell y_1,e_1),(y_2,e_2),(Sy_2,e_2),\dots)\\
 \text{ for some }(y_1,e_1),(y_2,e_2)\in Y\times \mathbb{A} \text{ and }\ell\geq 0.
\end{multline*}
Indeed, we can approximate any ``window'' $z[1,M]$ by $T^{n_j}[1,M]=x[n_j+1,n_j+M]$ and when $n_j\to \infty$, such a window has at most one point of ``discontinuity'', that  is,  it  contains  at  most  once  two  consecutive  coordinates  which are not successive images by $S\times I$ of some $(y_k,e_k)\in Y\times\mathbb{A}$. Thus, to conclude, we can use the same argument as in the first part \added{of the} proof.
\end{proof}

\begin{Remark}\label{r:nec}
Assume that $\bfu\colon\N\to\C$ is an arithmetic function \deleted{for}\added{relatively to} which the Sarnak property holds for each zero entropy $T$. We have already noticed that $\bfu$ has to satisfy~\eqref{eq:Mobius-like} but in fact, we obtain a stronger condition.

Fix $(X,T)$ of zero entropy.
We claim that we have the following arithmetic version of the strong MOMO property:
\begin{equation}\label{eq:momoarith}
\parbox{.85\textwidth}{for each $N\geq1$, $h=0,1,\ldots N-1$, $(b_k)\subset\N$ with $b_{k+1}-b_k\to\infty$,\\ $(x_k)$ and $f\in C(X)$, we have\\ $\frac1{b_K}\sum_{k<K}\left|\sum_{b_k\leq n<b_{k+1}}f(T^nx_k)\bfu(Nn+h)\right|\tend{K}{\infty}0$}
\end{equation}
Indeed, consider the $N$-discrete suspension $\widetilde{T}$ of $T$, i.e.\ $\widetilde{X}:=X\times\{0,1,\ldots,N-1\}$ and \added{let} the homeomorphism $\widetilde{T}$ act\deleted{s} by the formula $\widetilde{T}(x,j):=(x,j+1)$ when $0\leq j<N-1$ and $\widetilde{T}(x,N-1)=(Tx,0)$. Then $(\widetilde{X},\widetilde{T})$ has still zero entropy, and, by Corollary~\ref{cor:12},  the strong MOMO property is satisfied for $(\widetilde{X},\widetilde{T})$. Define $F\in C(\widetilde{X})$ by setting
$F(x,h)=f(x)$ and 0 otherwise. Hence
$$
\frac1{Nb_K}\sum_{k<K}\left|\sum_{Nb_k\leq n<Nb_{k+1}}F(\widetilde{T}^n(x_k,0))\bfu(n)\right|\tend{K}{\infty}0.$$
Therefore
$$
\frac1{Nb_K}\sum_{k<K}\left|\sum_{b_k\leq n<b_{k+1}}f(T^n(x_k))\bfu(Nn+h)\right|\tend{K}{\infty}0$$
and the claim~\eqref{eq:momoarith} follows.

In particular, the function $\bfu$ has to satisfy
\beq\label{eq:czymobius}
\frac1{b_K}\sum_{k<K}\left|\sum_{b_k\leq n<b_{k+1}}\bfu(Nn+h)\right|\tend{K}{\infty}0.
\eeq
Therefore, $\bfu$ is aperiodic\footnote{Note that  any non-principal Dirichlet character of modulus $q$ yields a (completely) multiplicative function for which we have~\eqref{eq:Mobius-like} (since the sum of the values along the period equals zero and $b_{k+1}-b_k\to \infty$) but which is not aperiodic.} and in fact, it is aperiodic on ``typical'' short interval:
\beq\label{eq:czymobius1}
\frac1M\sum_{M\leq m<2M}\left|\frac1H\sum_{m\leq g<m+H}\bfu(Ng+h)\right|\to 0\eeq
when $H\to\infty$ and $H/M\to0$.

Note that if above, for $X$ we take the one-point space, then $(\widetilde{X},\widetilde{T})$ stands for the rotation by 1 on $\Z/N\Z$ and the strong MOMO property
follows from~\cite{Ma-Ra-Ta}, see Subsection~\ref{s:MOMOds}. It follows that~\eqref{eq:czymobius} holds for $\bfu=\mob$.
\end{Remark}

\subsection{AOP implies~\ref{newP1},~\ref{newP2} and~\ref{P3} for M\"obius}
\label{sec:aop2P1}
We return now to the case where $\bfu\colon\N\to\C$ is a multiplicative function, $|\bfu|\leq1$, satisfying~\eqref{eq:Mobius-like}, in particular, we can take $\bfu=\mob$.

The purpose of this section is to prove that, in this setting, the AOP property ensures the validity of~\ref{newP1},~\ref{newP2} and~\ref{P3}.
For this, it will be useful to introduce the following weaker version of \ref{newP2} for the measure-theoretic dynamical system $(Z,\cd, \kappa,R)$, where we restrict the
class of continuous functions for which we demand that convergence~\eqref{eq:defSarnak} holds.

\setpropertytag{P2*}
\begin{property}\label{newP2*}
   For any topological dynamical system $(X,T)$ and any $x\in X$, if there exists a finite number of $T$-invariant measures $\mu_j$, $1\le j\le t$, such that
  \begin{itemize}
    \item for each $j$, $(X,\cb(X),\mu_j,T)$ is measure-theoretically isomorphic to $(Z,\cd,\kappa,R)$,
    \item  $\text{Q-gen}(x)\subset{\rm conv}(\mu_1,\ldots,\mu_t)$,
  \end{itemize}
  then, for any $f\in C(X)$ satisfying $\forall 1\le j\le t,\ \int_X f\,d\mu_j=0$, convergence~\eqref{eq:defSarnak} holds.
\end{property}

\begin{Prop}
  If $\bfu$ is a multiplicative function, $|\bfu|\leq1$, satisfying~\eqref{eq:Mobius-like}, then \ref{newP2*} $\Longrightarrow$\ref{P3}.
\end{Prop}
\begin{proof}
It follows by~\eqref{eq:Mobius-like} that if we want to prove the strong MOMO property in
  a specific uniquely ergodic model $(Y,S)$ with the unique invariant measure $\nu$,
  it is enough to check the required convergence  for a continuous function $f$ with $\int_Y f\,d\nu=0$. But then the continuous function $g$ constructed at the end of the proof of \ref{newP2} $\Longrightarrow$ \ref{P3} will satisfy $\int_X g\, d\mu_j=0$ for $j=0,1,2$ (see footnote~\ref{f:stopka1}), hence property~\ref{newP2*} will be enough.
\end{proof}

We will need the following criterion:

\begin{Prop} [KBSZ criterion, \cite{Ka,Bo-Sa-Zi}, see also \cite{Ab-Le-Ru1}] \label{p:kbsz}
Let $(a_n)\subset\C$ be bounded.
If
\beq\label{eq:kbsz0}
\limsup_{\mathscr{P}\ni r,s\to\infty,r\neq s}\limsup_{N\to\infty}\left|\frac1N\sum_{n< N}
a_{rn}\ov{a}_{sn}\right|=0\eeq
then $\lim_{N\to\infty}\frac1N\sum_{n< N}a_n\bfv(n)=0$ for each multiplicative function $\bfv\colon\N\to\C$, $|\bfv|\leq1$.\end{Prop}


\begin{Th}
  If $\bfu$ is a  multiplicative function, $|\bfu|\leq1$, satisfying~\eqref{eq:Mobius-like}, then AOP $\Longrightarrow$ \ref{newP2*}.
  In particular, the implication holds for the M\"obius function $\mob$.
\end{Th}

\begin{proof}
  Assume that $(Z,\cd,\kappa,R)$ enjoys the AOP property. Let $(X,T)$, $x$, $\mu_1,\ldots,\mu_t$ and $f$ be as in the assumptions of Property~\ref{newP2*}.
  In order to prove the convergence~\eqref{eq:defSarnak}, we want to apply the KBSZ criterion to the sequence ${\bigl(f(T^nx)\bigr)}_{n\ge0}$. Given $\varepsilon>0$, we have to show that, if $r\neq s$ are two different
  primes which are large enough, then
  \[
    \limsup_{N\to\infty}\frac{1}{N}\left| \sum_{n<N} f(T^{rn}x)\overline{f}(T^{sn}x)\right| < \varepsilon.
  \]
  Let $(N_i)$ be an increasing sequence of integers along which the convergence to the above limit superior holds. Without loss of generality, we can also assume that
  the sequence of empirical measures
  \[
    \sigma_{N_i}:=\frac{1}{N_i} \sum_{ n<N_i} \delta_{(T^{rn}x,T^{sn}x)}
  \]
  converges weakly to some $T^r\times T^s$-invariant measure $\rho$ on $X\times X$.
  We then have
 \[
    \limsup_{N\to\infty}\frac{1}{N}\left| \sum_{ n<N} f(T^{rn}x)\overline{f}(T^{sn}x)\right| = \left| \int_{X\times X} f\otimes \overline{f}\, d\rho\right|,
  \]
  and it is enough to prove that, if $r$ and $s$ are large enough, then for each ergodic component $\gamma$ of $\rho$, we have
  \[
     \left| \int_{X\times X} f\otimes \overline{f}\, d\gamma\right| < \varepsilon.
  \]

  Let $\rho_1$ (respectively $\rho_2$) be the marginal of $\rho$ on the first (respectively the second) coordinate. Then $\rho_1$ is $T^r$-invariant,
  and
  \[
    \frac{1}{r}\left( \rho_1 + T_*(\rho_1)+\cdots+T_*^{r-1}(\rho_1)\right) = \lim_{i\to\infty} \frac{1}{rN_i} \sum_{0\le n<rN_i} \delta_{T^nx}
  \]
is a $T$-invariant probability measure on $X$ for which $x$ is quasi-generic. By assumption, this measure is a convex combination of $\mu_1,\ldots,\mu_t$, and $\rho_1$ is absolutely continuous with respect to this convex combination. By the same argument, $\rho_2$ is also absolutely continuous with respect to some convex combination of $\mu_1,\ldots,\mu_t$.
Moreover, since $(Z,\cd,\kappa,R)$ has the AOP property, this system is totally ergodic, and by isomorphism, this also holds for each $(X,\cb(X),\mu_j,T)$. It follows that each ergodic component of $\rho$ is an ergodic joining of $(X,\cb(X),\mu_i,T^r)$ and $(X,\cb(X),\mu_j,T^s)$ for some $i,j\in\{1,\ldots,t\}$.

For each $j$, $1\le j\le t$, let $\varphi_j\colon (Z,\cd,\kappa,R)\to (X,\cb(X),\mu_j,T)$ be an isomorphism of measure-theoretic dynamical systems. Set also $f_j:=f\circ\varphi_j\in L^2(\kappa)$.
Let $\gamma\in J^e\bigl((X,\cb(X),\mu_i,T^r),(X,\cb(X),\mu_j,T^s)\bigr)$ be an ergodic component of $\rho$. Then the pushforward image $(\varphi_i\times\varphi_j)_*(\gamma)$ is an ergodic joining of $(Z,\cd,\kappa,R^r)$ and $(Z,\cd,\kappa,R^s)$, and we have
\begin{multline*}
  \left|\int_{X\times X}f\otimes \overline{f}\, d\gamma\right| = \left|\int_{Z\times Z}f_i\otimes \overline{f_j}\, d(\varphi_i\times\varphi_j)_*(\gamma)\right|\\
  \le \sup_{i,j\in\{1,\ldots,t\}} \sup_{\eta\in J^e(R^r,R^s)} \left|\int_{Z\times Z}f_i\otimes \overline{f_j}\, d\eta\right|.
\end{multline*}
But by the AOP property, if $r$ and $s$ are large enough, the modulus of the RHS is bounded by $\varepsilon$.
  \end{proof}

\subsection{Strong MOMO property in positive entropy systems}\label{s:momo}
In this section we prove Corollaries~\ref{p1j},~\ref{c11}, Corollary~\ref{dodENT} and equation~\eqref{m1}.

\begin{proof}[Proof of Corollary~\ref{p1j}]
 Suppose that the strong MOMO property (relative to $\bfv$) holds for $(X,T)$. Then~\ref{newP1}
from~\ref{thmAPm}
holds (for $\bfv$). Equivalently,~\ref{newP2}
holds.
Notice that if we take $\bfu$ as above then
the assumptions of~\ref{newP2}
are satisfied here (for $(\D_L^{\Z},S)$ and $\bfu$), whence the assertion of~\ref{newP2}
also holds, i.e.\ $\bfu$ satisfies the Sarnak property (relatively to $\bfv$). Take $g(y):=\overline{y(0)}$, and note that
$\frac1N\sum_{n\leq N}\bfu(n)\ov{\bfv(n)}=\frac1N\sum_{n\leq N}g(S^n\bfu)\ov{\bfv(n)}\to0$
which is a contradiction.
\end{proof}

Notice that Corollary~\ref{p1j} puts some restrictions on dynamical properties of measures for which $\bfu$ is  quasi-generic.

\begin{Cor}\label{c:aop}
Let $\bfu$ be a multiplicative arithmetic function, $|\bfu|\leq1$. Assume that $\text{Q-gen}(\bfu)\subset{\rm conv}(\kappa_1,\ldots,\kappa_m)$, where all $(X_{\bfu},\kappa_j,S)$ are isomorphic and have the AOP property. Then $\kappa_1=\ldots=\kappa_m=\delta_{(\ldots00.00\ldots)}$, i.e.\ $\bfu$ is equal to the zero sequence, up to a set of zero density.
\end{Cor}
\begin{proof}
Let $(X,T)$ be any uniquely ergodic model of $(X_{\bfu},\kappa_1,S)$. Then, since the AOP property is an isomorphism invariant, $(X,T)$ also satisfies the AOP property. Since this system is additionally uniquely ergodic, we have the strong MOMO property for $(X,T)$ and, by Corollary~\ref{p1j} (in which $\bfv=\bfu$), the proof is complete.
\end{proof}

In particular, if we consider $\lio$ (or $\mob$) then the set $\text{Q-gen}(\lio)$   cannot be contained in ${\rm conv}(\kappa_1,\ldots,\kappa_m)$, where the dynamical systems given by $\kappa_j$ are all isomorphic and have the AOP property.\footnote{The AOP property can be replaced by the existence of a uniquely ergodic model of the dynamical system associated to $\kappa_1$ for which we have the strong MOMO property.}
Studying properties of dynamical
systems given by (potentially many) measures in $\text{Q-gen}( \lio)$ is important in the
light of a recent remarkable result of N.\ Frantzikinakis \cite{Fr} which says that if
the logarithmic averages $\frac1{N_k}\sum_{n\leq N_k}\frac{\delta_{S^n\lio}}{n}$ converge to an ergodic measure then the logarithmic Chowla conjecture holds along the subsequence $(N_k)$. In particular, if
$\text{Q-gen}(\lio )$ consists solely of one measure and this measure is ergodic then the Chowla conjecture holds.

\begin{proof}[Proof of Corollary~\ref{c11}]
Assume that we have found a dynamical system $(X,T)$, $h(X,T)>h((\D_L)^{\Z},\kappa,S)$ which has the strong MOMO property relatively to $\bfv$. Then, there exists an ergodic $T$-invariant measure $\mu$ such that the entropy of the system $(X,\mu,T)$ is larger than $h((\D_L)^{\Z},\kappa,S)$. Next, by Sinai's theorem,
the Bernoulli automorphism $((\D_L)^{\Z},\kappa,S)$ is a measure-theoretic factor of $(X,\mu,T)$. By a theorem of B.\ Weiss \cite{We} there are uniquely (in fact, strictly) ergodic systems $(X',T')$, $(Y',S')$ (with the unique invariant measures $\mu'$ and $\nu'$, respectively) and a continuous, equivariant map $\pi'\colon X'\to Y'$ such that
$(X',\mu',T')$, $(Y',\nu',S')$ are measure-theoretically isomorphic to $(X,\mu,T)$ and $(X_{\bfu},\kappa,S)$, respectively.
Since $(X,T)$ has the strong MOMO property relative to $\bfv$, so has $(X',T')$ (by~\ref{thmAPm}).
Since $(Y',S')$ is a topological factor of
$(X',T')$, also $(Y',S')$ enjoys the strong MOMO property relatively to $\bfv$. But $(Y',S')$ is a uniquely ergodic model of $(X_{\bfu},\kappa,S)$, and we obtain a contradiction with Corollary~\ref{p1j}.
\end{proof}

\begin{proof}[Proof of Corollary~\ref{dodENT}]
The Chowla conjecture for $\lio$ means that $\lio$ is a generic point for the Bernoulli measure $B(1/2,1/2)$ on $\{-1,1\}^{\Z}$. In view of Corollary~\ref{c11}, all we need to prove is that we can find $\bfu\in\{0,1\}^{\Z}$ a generic point for a Bernoulli measure $\kappa$ of arbitrarily small entropy such that $\bfu$ correlates with $\lio$ (indeed, set $\bfv=\lio$ in Corollary~\ref{c11}).

Let $0<p<1$ and denote by $\rho$ the Bernoulli measure $B(p,1-p)$ on $\{0,1\}^{\Z}$. By Theorem~2.10 in \cite{Co-Do-Se}, it follows that we can find $y\in\{0,1\}^{\N}$ such that $(y,\lio)$ is a generic point for $\rho\ot B(1/2,1/2)$. Then set
\beq\label{spryt}
\bfu(n)=\lio(n)\text{ if }y(n)=1\text{, and }\bfu(n)=-1\text{ if } y(n)=0.\eeq
If $f\colon\{0,1\}\times \{-1,1\}\to\{-1,1\}$ is defined by
$$
f(0,a)=-1\text{ and }f(1,a)=a$$
and $F\colon\{0,1\}^{\Z}\times\{-1,1\}^{\Z}\to\{-1,1\}^{\Z}$
is given by
$$
F((x(n)),(a(n))_{n\in\Z})=(f(x(n),a(n)))_{n\in\Z}$$
then $\bfu$ is a generic point for the measure $F_\ast(\rho\ot B(1/2,1/2))$ and (by independence) the latter measure is the Bernoulli measure $\kappa=B(\frac{1+p}2,\frac{1-p}2)$ on $\{-1,1\}^{\Z}$ (which, by choosing  $p$ close to~1, has as small entropy as we need).

Finally, $\bfu$ and $\lio$ do correlate since $f(x,a)=a$ with probability $\frac12+\frac{1-p}2$, whence the frequency of 1s on $(\bfu(n)\lio(n))$ is at least $\frac12+\frac{1-p}2$.
\end{proof}


The following technical lemma  shows what happens if we contradict the condition from the definition of the strong MOMO property.
\begin{Lemma}\label{l1}
Let $f\in C(X)$, $(x_k)\subset X$, $(b_k)\subset\N$, $b_{k+1}-b_k\to\infty$. Assume that
$$
\limsup_{K\to\infty}\frac1{b_{K+1}}\sum_{k\leq K}\left|\sum_{b_k\leq n<b_{k+1}}f(T^nx_k)\bfu(n)\right|>0.
$$
Then there exist $\vep_0>0$ and a collection $\{ [a_k,c_k)\subset \N : k\geq 1\}$ of disjoint intervals with $c_k-a_k\to\infty$ and $\overline{d}(\bigcup_{k\geq 1}[a_k,c_k))>0$ such that
$$
\frac1{c_k-a_k}\left|\sum_{n=0}^{c_k-a_k-1}f(T^nx_k)\bfu(a_k+n)\right|\geq\vep_0\text{ for each }k\geq1.
$$
\end{Lemma}
\begin{proof}
We begin the proof by the following simple observation: for $0\leq F\in L^\infty\xbm $ and $\vep_0\leq\int F\,d\mu$, we have
$$
\vep_0\leq \int_{[F\leq \vep_0/2]}F\,d\mu+\int_{[F\geq\vep_0/2]}F\,d\mu\leq\vep_0/2+\|F\|_\infty\mu(\{x\in X:F(x)\geq \vep_0/2\}),
$$
whence
$$
\mu(\{x\in X:\: F(x)\geq \vep_0/2\})\geq \vep_0/(2\|F\|_\infty).
$$
Fix $(z_k)_{k\geq 1}$ and $(\alpha_k)_{k\geq 1}$ with $0\leq z_k\leq M$ and $\sum_{k\geq 1}\alpha_k=1$ and suppose that, for some $K\geq 1$, $\sum_{k\leq K}\alpha_kz_k\geq \vep_0$. Set $y_k:=z_k$ for $k\leq K$ and $y_k:=0$ for $k>K$. Then clearly $\sum_{k\geq 1}\alpha_ky_k=\sum_{k\leq K}\geq \vep_0$ and we can apply the above observation to obtain
\begin{equation}\label{m3j}
\sum_{k\leq K, z_k\geq \vep_0/2}\alpha_k=\sum_{k\leq K, y_k\geq \vep_0/2}\alpha_k\geq \vep_0/(2M).
\end{equation}
Let now $\vep_0>0$ and $(K_\ell)_{\ell\geq 1}$ be such that
\begin{equation}\label{m4j}
\frac1{b_{K_\ell+1}}\sum_{k\leq K_\ell}\left|\sum_{b_k\leq n<b_{k+1}}f(T^nx_k)\bfu(n)\right|\geq\vep_0.
\end{equation}
Set  $z_k:=\left|\frac1{b_{k+1}-b_k}\sum_{b_k\leq n<b_{k+1}}f(T^nx_k)\bfu(n)\right|$ and $\alpha_k:=\frac{b_{k+1}-b_k}{b_{K_\ell+1}}$. Then~\eqref{m4j} takes the form
$$
\sum_{k\leq K_\ell}\alpha_kz_k\geq \vep_0.
$$
Since
$$
\overline{d}\Big(\bigcup_{k\geq 1,z_k\geq \vep_0/2}[b_{k},b_{k+1}) \Big)\geq \limsup_{\ell\to \infty}\frac{1}{K_\ell}\sum_{k\leq K_\ell,z_k\geq \vep_0/2}(b_{k+1}-b_k),
$$
it follows from~\eqref{m3j} that
$$
\overline{d}\Big(\bigcup_{k\geq 1,z_k\geq \vep_0/2}[b_{k},b_{k+1}) \Big)\geq \vep_0/(2M).
$$
Thus, we obtain a sequence $[a_k,c_k)$, $k\geq1$, of disjoint intervals of the form $[b_{k_s},b_{k_s+1})$ such that
$$
\overline{d}\Big(\bigcup_{k\geq 1}[a_{k},c_{k}) \Big)\geq \vep_0/(2M)>0
$$
and
$$
\left|\!\frac1{c_{k}-a_k}\!\sum_{a_k\leq n<c_{k}}\!\!f(T^nx_k)\bfu(n)\!\right|\geq\vep_0/2.
$$
This completes the proof.
\end{proof}
Equation~\eqref{m1} immediately follows from the above lemma.

\section{Lifting strong MOMO to extensions}

In this section, we continue our considerations on orthogonality to bounded \deleted{by~1} multiplicative functions $\bfu$ satisfying~\eqref{eq:Mobius-like}; in particular, we can take $\bfu=\mob$.

\subsection{Lifting strong MOMO to Rokhlin extensions}
We start from a uniquely ergodic system $(X,T)$ enjoying the strong MOMO property, and we denote by $\mu_X$ the unique $T$-invariant probability measure. We consider a continuous extension $\bar T$ of $T$ to some product space $X\times Y$; $Y$ is also a compact metric space, and $\bar T$ is a continuous transformation of $X\times Y$ which has the form $\bar T(x,y)=(Tx,S_x(y))$. We also assume that $\bar T$ is uniquely ergodic, its unique invariant measure having the form $\mu_X\otimes \mu_Y$ for some probability measure $\mu_Y$ on $Y$.
Our purpose is to give a sufficient condition for the strong MOMO property to hold in the extension $(X\times Y,\bar T)$. The condition we give can be seen as a form of relative disjointness of $\bar T^r$ and $\bar T^s$ over the base system (for large different prime integers $r$ and $s$), and the proof relies on the same kind of arguments as in the proof of \ref{newP2} $\Longrightarrow$ \ref{P3}.

\begin{Th}\label{t:SMext}
  Suppose that, for all large enough prime numbers $r\neq s$, the following holds: each probability measure on $(X\times Y)\times (X\times Y)$ which is invariant and ergodic under the action of $\bar T^r\times \bar T^s$ is, up to a natural permutation of coordinates, of the form $\rho\otimes\mu_Y\otimes\mu_Y$, where $\rho$ is some $T^r\times T^s$-invariant measure on $X\times X$. Then the strong MOMO property also holds in $(X\times Y,\bar T)$.
\end{Th}

\begin{proof}
We fix  an increasing sequence of integers $0=b_0<b_1<b_2<\cdots$ with $b_{k+1}-b_k\to\infty$, and
a sequence of points ${\bigl((x_k,y_k)\bigr)}_{k\ge0}$ in $X\times Y$. We also fix a continuous function $f$ on $X\times Y$, and we assume that $f$ is of the form $f=f_1\otimes f_2: \ (x,y)\mapsto f_1(x)f_2(y)$ where $f_1\in C(X)$ and $f_2\in C(Y)$. Considering continuous functions of this type on $X\times Y$ is enough for our purposes, since they generate a dense subspace in $C(X\times Y)$ and $\bar T$ is uniquely ergodic. We thus have to prove the convergence
\[
  \frac{1}{b_K} \sum_{k<K}  \left| \sum_{b_k\le n<b_{k+1}}
  f_1\otimes f_2\left( \bar T^{n-b_k}(x_k,y_k) \right) \bfu(n) \right|  \tend{K}{\infty} 0.
\]
We observe that, subtracting  $\int_Y f_2\, d\mu_Y$ from $f_2$ if necessary,
which does not affect the above limit by the strong MOMO property of $(X,T)$, we can always assume that
\begin{equation}
  \label{eq:f_2_centered}
  \int_Y f_2\, d\mu_Y=0.
\end{equation}

We again consider the set $\A:=\{1,e^{i2\pi/3},e^{i4\pi/3}\}$ of third roots of unity, and for each $k$, we choose $e_k\in\A$ such that
\[
  e_k \left(  \sum_{b_k\le n<b_{k+1}}
  f_1\otimes f_2\left( \bar T^{n-b_k}(x_k,y_k) \right) \bfu(n)   \right)
\]
belongs to the closed convex cone $\{0\}\cup\{z\in\C:\arg(z)\in[-\pi/3,\pi/3]\}$.
Then again by Lemma~\ref{lemma:cone}, it is enough to prove that
\begin{equation}
  \label{eq:convergencetoprove2}
  \frac{1}{b_K} \sum_{ k<K}  \sum_{b_k\le n<b_{k+1}}  e_k f_1\otimes f_2\left( \bar T^{n-b_k}(x_k,y_k) \right) \bfu(n)  \tend{K}{\infty} 0.
\end{equation}

Now, we introduce the space $Z:=(X\times Y\times \A)^\N$, on which acts the shift map $S$.
We define a
particular point $z\in Z$ by setting $z_n:=( \bar T^{n-b_k}(x_k,y_k),e_k)$ whenever $b_k\le n<b_{k+1}$, and we consider the continuous function $F$ on $Z$ defined by
\[
  F\Bigl( (u_0,v_0,a_0),(u_1,v_1,a_1),\ldots  \Bigr) := a_0f_1(u_0)f_2(v_0).
\]
Then, to get~\eqref{eq:convergencetoprove2}, it is enough to establish the orthogonality of $\bfu$ and $\left(F(S^nz)\right)$, i.e.
\begin{equation}
  \label{eq:convergencetoprove3}
  \frac{1}{N} \sum_{n< N} F(S^nz) \bfu(n) \tend{N}{\infty} 0.
\end{equation}
Using the KBSZ criterion, the above holds as soon as, for all large enough different prime numbers $r$ and $s$, we have
\begin{equation}
  \label{eq:convergencetoprove4}
  \frac{1}{N} \sum_{ n< N} F(S^{rn}z) \overline{F}(S^{sn}z)  \tend{N}{\infty} 0.
\end{equation}

Let $(N_i)$ be an increasing sequence of positive integers along which the sequence of empirical measures
\[
  \frac{1}{N_i} \sum_{n<N_i} \delta_{(S^{rn}z,S^{sn}z)},\;i\geq1,
\]
converges  to some $S^r\times S^s$-invariant probability measure $\mu_{Z\times Z}$ on $Z\times Z$.
We therefore have
\[
  \frac{1}{N_i} \sum_{ n< N_i} F(S^{rn}z) \overline{F}(S^{sn}z)  \tend{i}{\infty} \int_{Z\times Z} F\otimes \overline{F}\, d\mu_{Z\times Z}.
\]

Since $b_{k+1}-b_k\to\infty$, the measure $\mu_{Z\times Z}$ is concentrated on the set of pairs $(w,w')\in Z^2$ whose coordinates are of the form
\[ w=\Bigl( (u,v,a),(\bar T(u,v),a), (\bar T^2(u,v),a),\ldots\Bigr)\]
and
\[ w'=\Bigl( (u',v',a'),(\bar T(u',v'),a'), (\bar T^2(u',v'),a'),\ldots\Bigr)\]
for some $u,u'$ in $X$, some $v,v'$ in $Y$, and some $a,a'$ in $\A$.
By the invariance of $\mu_{Z\times Z}$ under $S^r\times S^s$, the marginal of $\mu_{Z\times Z}$ given by the coordinates $\bigl((u,v),(u',v')\bigr)$ is $\bar T^r \times \bar T^s$-invariant. Hence, if instead of $\mu_{Z\times Z}$ we consider an ergodic component $\gamma$ of $\mu_{Z\times Z}$, the assumption of the theorem implies that the marginal given by the coordinates $\bigl((u,v),(u',v')\bigr)$ is of the form $\rho\otimes\mu_Y\otimes\mu_Y$ for some probability measure $\rho$ on $X\times X$. Moreover, using again the disjointness of ergodicity and identity, we see that under $\gamma$ the coordinates $(a,a')$ are independent of $\bigl((u,v),(u',v')\bigr)$. If we denote by $\gamma_{\A\times \A}$ the marginal defined on $\A\times\A$ by the coordinates $(a,a')$, we thus have
\begin{multline*}
\int_{Z\times Z} F\otimes \overline{F}\, d\gamma = \int_{\A\times \A} a\overline{a'} d\gamma_{\A\times A}(a,a') \int_{X\times X} f_1(u)\overline{f_1}(u')\,d\rho(u,u')
 \\
\times \int_Y f_2(v)\, d\mu_Y(v)\int_Y \overline{f_2}(v')\, d\mu_Y(v'),
\end{multline*}
hence this integral vanishes by~\eqref{eq:f_2_centered}. Since this is true for all ergodic components of $\mu_{Z\times Z}$, we get
\[
  \lim_{i\to\infty}\frac{1}{N_i} \sum_{ n< N_i} F(S^{rn}z) \overline{F}(S^{sn}z)  = \int_{Z\times Z} F\otimes \bar F d\mu_{Z\times Z}=0,
\]
so~\eqref{eq:convergencetoprove4} follows and the proof is complete.
\end{proof}

\subsection{Sarnak's conjecture for continuous extensions by coboundaries}
In the theorem below we consider homeomorphisms for which the measure-theoretic systems determined by ergodic invariant measures are all isomorphic. However, the set of ergodic invariant measures is uncountable which makes a direct use of~\ref{thmAPm} questionable. On the other hand, this set is quite structured which allows one to repeat the main steps of the proof of the implication  \ref{newP1} $\Longrightarrow$ \ref{newP2}.

\begin{Th}\label{t:meascob}
Suppose that $(Y,S)$ is uniquely ergodic and satisfies the strong MOMO property [relatively to $\bfu$]. Let $G$ be a compact Abelian group (with Haar measure $\la_G$) and let $\varphi\colon Y\to G$ be continuous with $\varphi=\psi\circ S-\psi$, where $\psi\colon Y\to G$ is measurable. Then $(Y\times G,S_\varphi)$, $S_\varphi(x,g)=(Sx,\varphi(x)+g)$, satisfies the Sarnak property [relatively to $\bfu$].
\end{Th}
\begin{proof}
First, we need to introduce some notation. For $g\in G$, let $\widetilde{A}_g$ be the graph of $\psi+g$, i.e.\ $\widetilde{A}_g:=\{(y,\psi(y)+g): y\in Y\}$ and let $\pi_g\colon \widetilde{A}_g\to Y$ stand for the projection onto the first coordinate. Let $\widetilde{\nu}_g:=\left(\pi^{-1}_g\right)_\ast\nu$, where $\nu$ is the unique invariant measure for $(Y,S)$. Then the ergodic decomposition of the product system $(Y\times G,\nu\ot\la_G, S_\varphi)$ is given by
$$
\nu\otimes \lambda_G=\int \widetilde{\nu}_g\, d\lambda_G(g).
$$
But what is more important here is that  $\{ \widetilde{\nu}_g:\:g\in G\}$ is the set of $S_\varphi$-invariant ergodic measures, see e.g.\ \cite{Ke-Ne}.
Since $(Y,\cb(Y),\nu,S)\simeq (Y\times G,\cb(Y\times G),\widetilde{\nu}_0,S_\varphi)$ via the map $Id \times \psi \colon Y\to Y\times G$, it follows by Lusin's theorem that there exists a compact set $K\subset Y$ such that $\nu(K)>1-\vep^4$ and such that the restriction of $Id\times \psi$ to $K$ is continuous. Then $\widetilde{K}_0:=(Id\times \psi)(K)\subset \widetilde{A}_0$ is also compact. Define $\widetilde{K}_g:=\widetilde{K}_0+g$ and notice that $\bigcup_{g\in G}\widetilde{K}_g=K\times G$.

Now, fix $(\overline{y},\overline{g})\in Y\times G$. We will show that the Sarnak property holds in $(\overline{y},\overline{g})\in Y\times G$ by showing convergence~\eqref{eq:defSarnak} for functions of the form $F=f\otimes \chi$, where $f\in C(Y)$ and $\chi\in\widehat{G}$ is a character on $G$.\footnote{We recall that such functions form a set which is linearly dense in the uniform topology in $C(Y\times G)$.} Let $H_0$ be a continuous extension of $F\circ \pi_0^{-1}|_K$ to the whole space $Y$, such that $\|H_0\|_\infty= \|F\|_\infty$ (such $H_0$ exists by the Tietze extension theorem).  Let $H_h:=\chi(h)H_0$ for $h\in G$. Then $H_h$ is a continuous extension of $F\circ \pi_h^{-1}|_K$. Indeed, for $y\in K$, we have
\begin{multline*}
H_h(y)=\chi(h)H_0(y)=\chi(h)F\circ \pi_0^{-1}(y)=\chi(h)F(y,\psi(y))\\
=\chi(h)f(y)\chi(\psi(y))=\chi(h+\psi(y))f(y)=F(y,\psi(y)+h)=F\circ \pi_h^{-1}(y).
\end{multline*}
Notice that
\begin{multline}\label{eq:window_cob_h}
\text{if }(y,g)\in \widetilde{K}_h\text{ and }S_\varphi^s(y,g)\in \widetilde{K}_h\text{ then }\\
F(S_{\varphi}^s(y,g))=(F\circ \pi_h^{-1})( \pi_h(S_\varphi^s(y,g)))\\
=H_h(\pi_h(S_\varphi^s(y,g)))=H_h(S^s(y)).
\end{multline}

For $L\geq 1$, define the following compact subset of $\widetilde{K}_0$:
$$
B_0(L):=\left\{(y,g) \in \widetilde{K}_0 : \frac{1}{L}\sum_{l\leq L}\ind{\widetilde{K}_0}(S_\varphi^s(y,g))>1-\vep^2\right\}.
$$
In the same way, we define $B_h(L)$ for $h\in G$. It follows by a straightforward calculation that $B_h(L)=B_0(L)+h$ for each $h\in G$. Finally, define $B(L):=\bigcup_{h\in G}B_h(L)$. Clearly, $B(L)=\pi_0(B_0(L))\times G$, whence it is again a compact set. By repeating an argument from the proof of \ref{newP1} $\Longrightarrow$ \ref{newP2}, we obtain
$\widetilde{\nu}_h(B_h(L))>1-\vep$ for each $h\in G$.

Let $D(\cdot,\cdot)$ be the product distance on $Y\times G$. For each $L\geq 1$, define $\eta(L)>0$ such that for $(y,g),(y',g')\in Y\times G$, we have
\begin{multline}
\label{eq:equicontinuity_cob}
  D((y,g),(y',g'))<\eta(L)\Longrightarrow\\
|F(S_\varphi^n(y,g))-F(S_\varphi^n(y',g'))| < \varepsilon \text{ for all }0\le n<L.
\end{multline}

By repeating the proof of Lemma~\ref{lemma:density}, we obtain that
\begin{equation}\label{eq:density_cob}
\limsup_{N\to \infty}\frac{1}{N}\#\{0\leq n\leq N-1 : D(S_\varphi^n(y,g),B(L))\geq \eta(L)\}<\vep.\footnote{Indeed, if $(\overline{y},\overline{g})$ is quasi-generic, along a sequence $(N_i)$, for an $S_\varphi$-invariant measure then this measure has to be of  the form $\widetilde{\mu}=\int_G {\widetilde{\nu}_g}\, dP(g)$, where $P$ is a probability measure on $G$.}
\end{equation}

We now fix an increasing sequence of integers $1\le L_1<L_2<\cdots$. Repeat the arguments from the proof of \ref{newP1} $\Longrightarrow$ \ref{newP2}
to obtain sequences $(M_i)_{i\geq0}$ and $(b_k)_{k\geq0}$.

Now, fix $i\ge1$. Let $K\ge1$ be the largest integer such that $b_{K}\le M_i$. We want to approximate
the sum
\[ S_{M_i}:=\sum_{n<M_i} F(S_{\varphi}^n(\overline{y},\overline{g})) \bfu(n)\]
by the following expression coming from the dynamical system $(Y,S)$:
\[ E_K:=\sum_{ k<K} \sum_{b_k\le n<b_{k+1}} H_{h_k}  (S^{n-b_k} y_k) \bfu(n).\]
As in the proof of \ref{newP1} $\Longrightarrow$ \ref{newP2}, we obtain that
\begin{multline*}
   \limsup_{N\to\infty}\frac{1}{N} \left|\sum_{ n<N} F(S_\varphi^n(\overline{y},\overline{g}))\bfu(n)\right|
   =\lim_{i\to \infty} \frac{1}{M_i} \left|S_{M_i}\right|
   \le \limsup_{K\to\infty} \frac{1}{b_{K}}|E_K| + C\varepsilon
\end{multline*}
for some constant $C>0$. However
\begin{align*}
  \frac{1}{b_{K}}|E_K| &\le
  \frac{1}{b_{K}} \sum_{k<K} \left| \sum_{b_k\le n<b_{k+1}} H_{h_k}  (S^{n-b_k} y_k) \bfu(n) \right| \\
  & =   \frac{1}{b_{K}} \sum_{k<K} \left| \sum_{b_k\le n<b_{k+1}}\chi(h_k) H_0 (S^{n-b_k} y_k) \bfu(n) \right|\\
  &=   \frac{1}{b_{K}} \sum_{ k<K} \left| \sum_{b_k\le n<b_{k+1}} H_0 (S^{n-b_k} y_k) \bfu(n) \right|
\end{align*}
which goes to 0 as $K\to\infty$ by the strong MOMO property of $(Y,S)$. This completes the proof.
\end{proof}

\begin{Remark} \em See \cite{Li-Sa}, where there is the first  example of a continuous (in fact analytic) $f\colon\T\to\R$  considered over an irrational rotation $T$ such that $T_{e^{2\pi if}}$ is minimal, $f$ is a measurable coboundary and $T_{e^{2\pi if}}$ is M\"obius disjoint.  Cf.\ also~\cite{Wa} for M\"obius disjointness of all analytic Anzai skew products.
\end{Remark}

\section{Examples}
\subsection{Ergodic systems with discrete spectrum} \label{s:MOMOds} Recall that in \cite{Ab-Le-Ru1} it has been proved that all uniquely ergodic models of totally ergodic systems with discrete spectrum are M\"obius disjoint. The result has been extended to all uniquely ergodic models of all ergodic systems with discrete spectrum in \cite{Hu-Wa-Zh}.
By Halmos-von Neumann theorem it follows that if $(Z,\cd,\kappa,R)$ is an ergodic transformation with discrete spectrum then one of its uniquely ergodic models is a rotation
$Tx=x+x_0$, where $X$ is a compact, metric group and $\ov{\{nx_0:n\in\Z\}}=X$. If $\chi\in\widehat{X}$, $(b_k)\subset\N$ with $b_{k+1}-b_k\to\infty$ and $(x_k)\subset X$ then for each $(y_k)\subset X$, we have
$$
\frac1{b_K}\sum_{k<K}\sum_{b_k\leq n<b_{k+1}}\chi(T^n(x_k+y_k))\bfu(n)=\frac1{b_K}\sum_{k<K}\chi(y_k)\sum_{b_k\leq n<b_{k+1}}\chi(T^nx_k)\bfu(n).$$
It easily follows from Lemma~\ref{lemma:cone} that in $(X,T)$ we have the strong MOMO property whenever we have the MOMO property\deleted{there}. But
\beq\label{e:analiza}
\left|
\frac1{b_K}\sum_{k<K}\sum_{b_k\leq n<b_{k+1}}\chi(T^n(x_k))\bfu(n)\right|\leq
\frac1{b_K}\sum_{k<K}\left|\sum_{b_k\leq n<b_{k+1}}\left(\chi(x_0)\right)^n\bfu(n)\right|.\eeq
We have $\chi(x_0)=e^{2\pi i \alpha}$ for a unique $\alpha\in[0,1)$.
Two cases now arise:

{\bf Case 1.} If $\alpha$ is irrational then the RHS in~\eqref{e:analiza} goes to zero by \cite{Ab-Le-Ru1} for any $\bfu$ multiplicative, $|\bfu|\leq1$.

{\bf Case 2.} Assume that $\alpha$ is rational. Then it follows from Theorem~1.7 in \cite{Ma-Ra-Ta} that the RHS in~\eqref{e:analiza} goes to zero if $\bfu$ is multiplicative, $|\bfu|\leq1$ and
\beq\label{e:mrt}
\inf_{|t|\leq M,\xi~\text{mod}~q,q\leq Q}D(\bfu,n\mapsto \xi(n)n^{it};M)^2\to \infty,
\eeq
when  $10\leq H\leq M$, $H\to\infty$ and $Q=\min(\log^{1/125}M,\log^5 H)$; here $\xi$ runs over all Dirichlet characters of modulus $q\leq Q$ and
$$
D(\bfu,\bfv;M):=\left(\sum_{p\leq M, p\in\mathscr{P}}\frac{1-{\rm Re}(\bfu(p)\overline{\bfv(p)})}{p}\right)^{1/2}
$$
for each $\bfu,\bfv:\N\to\C$ multiplicative satisfying $|\bfu|,|\bfv|\leq1$.
In particular,~\eqref{e:mrt} implies~\eqref{eq:Mobius-like}. Moreover, classical multiplicative functions like $\mob$ and $\lio$ satisfy~\eqref{e:mrt} \cite{Ma-Ra-Ta}.

\begin{Cor}\label{c:SMds} Let $(Z,\cd,\kappa,R)$ be \deleted{a totally}\added{an} ergodic system with discrete spectrum. \added{If $R$ is totally ergodic,} then in each uniquely ergodic model of $R$ we have the strong MOMO property relatively to any  multiplicative function $\bfu$, $|\bfu|\leq1$, satisfying~\eqref{eq:Mobius-like}.
If \deleted{$R$ is ergodic and its spectrum}\added{the spectrum of $R$} has a nontrivial rational eigenvalue then the strong MOMO property holds for any $\bfu$ satisfying additionally~\eqref{e:mrt}.\end{Cor}

\subsection{Systems satisfying the AOP property}
\subsubsection{Systems whose powers are disjoint. Typical systems}
Each totally ergodic transformation whose prime powers are disjoint satisfies the AOP property. This includes large classes of rank one transformations: \cite{Ab-Le-Ru,Bo1,Ry}, automorphisms with the minimal self-joining property~\cite{Ju-Ru}, and recently it has been shown by Chaika and Eskin \cite{Ch-Es} that a.e.\ 3-interval transformation has sufficiently many disjoint prime powers. By Theorem~\ref{thmB}, it follows that the first assertion of Corollary~\ref{c:SMds} holds for an arbitrary $R$ belonging to any of those classes of transformations. In particular, by \cite{Ju}, it follows that that all uniquely ergodic models of a typical system $R$ of a standard Borel probability space satisfies this assertion.
\subsubsection{Unipotent diffeomorphisms on nilmanifolds}
As shown in \cite{Fl-Fr-Ku-Le}, ergodic unipotent diffeomorphisms  $T(x\Gamma)=uA(x)\Gamma$, where
$G$ is a connected, simply connected, nilpotent
Lie group, $\Gamma\subset G$ a lattice, $A\colon G\to G$ is a unipotent automorphism with $A(\Gamma)=\Gamma$ and $u\in G$, enjoy the AOP property. It follows that all other uniquely ergodic models of such systems satisfy the strong MOMO property (relatively to a bounded multiplicative $\bfu$ satisfying~\eqref{eq:Mobius-like}), in particular, this holds for uniquely ergodic models of ergodic nilrotations. A special case of unipotent diffeomorphisms on nilmanifolds are affine automorphisms on Abelian compact connected groups. In particular, we obtain that in all uniquely ergodic models of quasi-discrete spectrum systems, we have the strong MOMO property (with respect to $\bfu$), cf.\ \cite{Ab-Le-Ru1}.

A general unipotent case seems to be much less clear. We have been unable to answer the following:
\begin{Question}
Does, for each horocycle flow on the unit tangent bundle of compact surfaces of constant negative curvature, the time-1 automorphism  satisfy the (strong) MOMO property (relatively to $\mob$)?
(For M\"obius disjointness of time automorphisms of horocycle flows, see \cite{Bo-Sa-Zi}; such automorphisms do not possess the AOP property \cite{Ab-Le-Ru1}.)
\end{Question}



\subsubsection{Cocycle extensions of irrational rotations} As we have already noticed, all irrational rotations have the AOP property. The extensions of them considered in this subsection satisfy the assumptions of Theorem~\ref{t:SMext}, but in fact we lift AOP.

Consider $f\colon\T\to\R$ which is $C^{1+\delta}$ and which is not a trigonometric polynomial. In \cite{Ku-Le}, it is shown that for a $G_\delta$ and dense set of irrational $\alpha$ the corresponding Anzai skew product $T_{e^{2\pi if}}$ on $\T\times\bs^1$: $(x,z)\mapsto (x+\alpha,e^{2\pi if(x)}\cdot z)$ enjoys the AOP property (cf.\ Corollary~2.5.6 in \cite{Ku-Le}). Moreover, it is proved (again for a ``typical'' $\alpha$) in \cite{Ku-Le1} that the Rokhlin skew product $T_{f,\cs}$: $(x,y)\mapsto (x+\alpha,S_{f(x)}(y))$  enjoys the AOP property for each ergodic flow $\cs=(S_t)_{t\in\R}$ acting on $\ycn$  (on $X\times Y$ we consider product measure $\mu\ot\nu$), see Proposition~5.1 and Corollary~5.2 in \cite{Ku-Le1}.

Consider an affine case: $f(x)=x-\frac12$. Theorem~7.10 in \cite{Ku-Le1} tells us that $T_{f,\cs}$ has  the AOP property whenever $\alpha$ has bounded partial quotients and the spectrum of the flow $\cs$ on $L^2_0\ycn$ does not contain any rational number.

Note that if we replace $f$ by $f':=f+j-j\circ T$, where $j\colon \T\to\R$ is continuous, then the resulting skew products $T_{e^{2\pi if'}}$ or $T_{f',\cs}$\footnote{Formally, we should slightly extend $T$ using so-called Sturmian models \cite[Chapter~6 by Arnoux]{Arnoux}, so that $f$ becomes continuous, see \cite{Ku-Le1} for details.} are uniquely ergodic models of $T_{e^{2\pi if}}$ and $T_{f,\cs}$ (whenever $\cs$ is a uniquely ergodic flow), respectively. In particular,  $T_{e^{2\pi if'}}$ and $T_{f',\cs}$ enjoy the strong MOMO property (relatively to $\bfu$ satisfying~\eqref{eq:Mobius-like}).

\begin{Question}
Assume that $T_{e^{2\pi if}}$ is an ergodic Anzai skew product with $f\colon\T\to\R$ analytic. Does $T_{e^{2\pi if}}$ enjoy the strong MOMO property, the AOP property?
(It has been shown by Wang \cite{Wa} that all such skew products are M\"obius disjoint.)
\end{Question}

\subsection{Cocycles extensions of odometers. Morse and Kakutani sequences}
When $T$ is an odometer then it is not totally ergodic, and the method of AOP fails. However, we have already shown that odometers satisfy the strong MOMO property.
In this section, we will give examples of extensions of odometers illustrating Theorem~\ref{t:SMext}. Because of the second assertion in Corollary~\ref{c:SMds}, we will constantly assume that $\bfu$ \added{is a bounded multiplicative arithmetic function which} satisfies~\eqref{e:mrt}.
\added{(In particular, the following results are valid when $\bfu$ is the M\"obius or the Liouville function.)}

\subsubsection{Odometers, Toeplitz systems and generalized Morse systems}
Given an increasing sequence $(n_t)$ of natural numbers with $n_t|n_{t+1}$,  $t\geq0$ ($n_0=1$) set  $\la_t:=n_{t+1}/n_t$ for $t\geq0$ and let
$X=\prod_{t\geq0}\Z/\la_t\Z$. It is a compact, metric and monothetic group where the addition is coordinatewise with carrying the remainder to the right.
If by $\mu_X$ we denote Haar measure on $X$ then $(X,\cb(X),\mu_X,T)$, where $Tx=x+\underline{1}$, $\underline{1}:=(1,0,0,\ldots)$, is  ergodic (in fact, $(X,T)$ is uniquely ergodic). If
\beq\label{eq:part}
D^{(t)}=\{D^{(t)}_0,D^{(t)}_1,\ldots,D^{(t)}_{n_t-1}\},
\eeq
where $D^{(t)}_{0}=\{x\in X: x_0=\ldots=x_{t-1}=0\}$, $D^{(t)}_j=T^jD^{(t)}_0$, $j=0,1,\ldots, n_t-1$ then $D^{(t)}$ is a partition of $X$ and $\bigcup_{j=0}^{n_t-1}D^{(t)}_j=X$, that is, $D^{(t)}$ is a Rokhlin tower \deleted{ful}filling the whole space.

Let now $b^t\in\{0,1\}^{\la_t}$, $b^t(0)=0$, for $t\geq0$. The sequence
\beq\label{eq:morse}
x:=b^0\times b^1\times \ldots~\footnote{The multiplication of blocks $b^0,b^1,\ldots$ is from the left to the right; $B\times C:=(B+C(0))(B+C(1))\ldots (B+C(\la-1))$, where $\la:=|C|$ stands for the length of $C$ and $B+c:=(B(0)+c)(B(1)+c)\ldots (B(|B|-1)+c)$ (the addition mod~2 on each coordinate).} \eeq
is called a {\em generalized Morse sequence} \cite{Ke}.\footnote{As a matter of fact, there are some mild assumptions on the sequence $(b^t)$ to obtain a non-trivial dynamical systems, see e.g.\ the concept of continuous Morse sequence in \cite{Ke}.}  \added{Let $\widehat{x}$ be the sequence defined by}
\beq\label{eq:toeplitz}
\added{\widehat{x}(n):=x(n)+x(n+1) \bmod 2\quad(n\ge0).}
\eeq
Then $\widehat{x}$ is a Toeplitz sequence \cite{Ja-Ke}, i.e.\ for each $j$ there is $k_j$ such that $\widehat{x}(j)=\widehat{x}(j+mk_j)$ for each $m\geq0$. Let $X_x$ and $X_{\widehat{x}}$ stand for the (two-sided) subshifts determined by $x$ and $\widehat{x}$, respectively. By \cite{Wi}, the odometer $(X,T)$ is the maximal equicontinuous factor of $(X_{\widehat{x}},S)$, where $S$ stands for the left shift. We will constantly assume that $\widehat{x}$ is regular \cite{Ja-Ke}. In this case $(X_{\widehat{x}},S)$ is a uniquely ergodic model of $(X,T)$ \cite{Ja-Ke}. Therefore, the strong MOMO property holds for $(X_{\widehat{x}},S)$. On the other hand, clearly, $(X_{\widehat{x}},S)$ is a topological factor
of $(X_x,S)$, and moreover, there is a topological isomorphism of $(X_x,S)$ with $(X_{\widehat{x}}\times(\Z/2\Z),S_\va)$, where $\va(z)=z(0)$ for each $z\in X_{\widehat{x}}$ \cite{Le}.
The cocycle $\va$ has a special form (see the notion of Morse cocycle below), and we will show that in certain classical cases the assumptions of Theorem~\ref{t:SMext} are satisfied.

\subsubsection{A little bit of algebra}
Denote $\Z_N=\Z/N\Z$ and let $0\neq s\in\Z_N$, $(s,N)=1$. Then $s\in\Z_N^\ast$, i.e.\ $s$ is in the group of invertible (under multiplication) elements in
the ring $\Z_N$. Therefore, we can write $\frac1s$ for the inverse of $s$ in $\Z_N$, \added{and for any integer $r$, $\frac rs$ is well defined as an element of $\Z_N$}.

In the ring $\Z_{N}$, consider the F-norm $\|i\|:=\min(i,N-i)$.




\begin{Lemma}\label{l:a2}
 Assume that $r,s,k\geq1$ are fixed, pairwise coprime, and let $(n_t)$ be an increasing sequence of integers. We assume that $(s,n_t)=1$ while $k|n_t$ for each $t\ge1$. Then there exists $\eta>0$ such that
$$
\liminf_{t\to\infty} \min_{0\leq j<k-1}\left\|\frac rs-j\frac{n_t}k\right\|\geq \eta n_t.$$
\end{Lemma}

\begin{proof}
Denote $b_t:=\frac rs\in\Z_{n_t}$, i.e.\
\beq\label{a1} r=sb_t\; {\rm mod}~n_t,\eeq
and we also interpret $b_t$ as an integer in $\{0,\ldots,n_t-1\}$. Fix $0<\vep<\frac1{ks}$, and suppose that
$b_t\in(j\frac{n_t}k-\vep n_t,j\frac{n_t}k+\vep n_t)$ for some $0\le j\le k$. Then
$$
sb_t\in \left(sj\frac{n_t}k-s\vep n_t,sj\frac{n_t}k+s\vep n_t\right).
$$

If $j=0$, then $sb_t\in (0,s\vep n_t)$ and~\eqref{a1} is really $r=sb_t$ as soon as $n_t>r$, whence $s|r$, a contradiction.

If $0<j<k$, then the number $sj\frac{n_t}k$ is of the form $\ell n_t+ \frac{j'}k n_t$ with $0<j'<k$ (remember that $(k,s)=1$),
and by~\eqref{a1} we also have
$$r\in \left(sj'\frac{n_t}k-s\vep n_t,sj'\frac{n_t}k+s\vep n_t\right).$$
In particular, $r> \frac{n_t}{k}-\vep s n_t=n_t\left(\frac1k-\vep s\right)$, which is impossible when $n_t$ is large.

If $j=k$, then $sb_t\in \bigl(sn_t(1-\vep),sn_t\bigr)$. By~\eqref{a1}, there exists an integer $\ell$ such that
$$ r\in  \left(\ell n_t-s\vep n_t,\ell n_t\right).$$
As $r\ge1$, we must have $\ell\ge1$, but then $r\ge n_t(1-s\vep)$, which is impossible when $n_t$ is large.
\end{proof}
%
%



\begin{Prop} \label{p:a1} Assume that $(X,T)$ is an odometer, $Tx=x+\underline{1}$ (with $\underline{1}=(1,0,0,\ldots)$) on $X$, where $X$ is determined by the numbers $n_t|n_{t+1}$, $t\geq0$. Assume that $r,s,k\geq1$ are fixed, pairwise coprime, and that $(r,n_t)=1=(s,n_t)$ while $k|n_t$ for each $t\geq1$.
\deleted{Assume  that $Wx=x+r/s$. Then, there exists a unique sequence $(b_t)$, $0\leq b_t<n_t$, $b_{t+1}=b_t$ mod~$n_t$ for each $t\geq1$ such that  $WD^{(t)}_0=T^{b_t}D^{(t)}_0=D^{(t)}_{b_t}$.}
Then, there exists a unique automorphism $W$ of $(X,\mathcal{B}(X),\mu_X)$ such that $W^r=T^s$. More precisely, there exists a sequence $(b_t)$, $b_t\in\Z_{n_t}$,  depending only on $r,s$ and $n_t$, such that
\beq
\label{eq:defW}
\forall t\ge 0,\forall 0\le i<n_t,\quad WD^{(t)}_i = D^{(t)}_{i+b_t}.
\eeq
(Here the addition has to be understood in $\Z_{n_t}$.)
Moreover, there exists $\eta>0$ such that
$$
\liminf_{t\to\infty} \sup_{0\le j\le k-1} \frac{\left\|b_t-j\frac{n_t}k\right\| }{n_t}\geq\eta.
$$
\end{Prop}
\begin{proof}
Let $W$ satisfy $W^s=T^r$. Since $(r,n_t)=1$ for each $t$, $T^r$ is ergodic, and the subsets $D_i^{(t)}$, $0\le i<n_t$ are also the levels of a Rokhlin tower of height $n_t$ for $T^r$ (but in a different order). $W$ commutes with $T^r$, hence the sets $WD_i^{(t)}$, $0\le i<n_t$ also form  a Rokhlin tower of height $n_t$ for $T^r$. By ergodicity of $T^r$, there exists only one such Rokhlin tower (up to cyclic permutation of the levels). We thus conclude that $W$ permutes the levels $D_i^{(t)}$ of the Rokhlin tower. In particular there exists $b_t\in\{0,\ldots,n_t-1\}$ such that $W D_0^{(t)} = D_{b_t}^{(t)}$. Using again the fact that $(r,n_t)=1$, we see that there exists an integer $m$ such that $T^{mr}=T$ on the finite $\sigma$-algebra generated by the sets $D_i^{(t)}$, $0\le i<n_t$. As $W$ commutes with $T^r$, $W$ also commutes with $T$ on this $\sigma$-algebra, thus~\eqref{eq:defW} holds.
Now the relation $W^s=T^r$ on this $\sigma$-algebra just says that $b_t=\frac rs$ in $\Z_{n_t}$.

Conversely, setting $b_t:=\frac rs$ in $\Z_{n_t}$ for each $t\ge0$, (which exists because $(s,n_t)=1$), we see that, for each $t\ge0$, $b_{t+1}=b_t\bmod n_t$ as $n_t|n_{t+1}$. Thus we can define a unique automorphism $W$ by~\eqref{eq:defW} (note that $\mathcal{B}(X))$ is the supremum of the increasing sequence of finite $\sigma$-algebras generated by the sets $D_i^{(t)}$). We then get an $s$-th root of $T^r$.

Finally, the existence of $\eta$ follows from Lemma~\ref{l:a2}.

\deleted{Since, by ergodicity, there exists only one Rokhlin tower of height $n_t$ \deleted{ful}filling the whole space (up to cyclic permutation of the levels) and $WD^{(t)}$ is such a tower, the numbers $b_t$ are well defined. The \deleted{result}\added{existence of $\eta$} follows from Lemma~\ref{l:a1}.}\footnote{The action of $W$ on the tower $D^{(t)}$ is the ``rotation'' by $b_t$;  $sb_t$ corresponds to $W^s$ which is $r$ mod~$n_t$ (which corresponds to $T^r$).}
\end{proof}

Note that the conclusion of the preceding proposition also implies
\beq\label{a6}
\left\|(n_t-b_t)-i\frac{n_t}k\right\|\geq\eta n_t\eeq
for all $t\geq1$.

We will say that an odometer $(X,T)$ has {\em small rational spectrum} if
the set
$$
{\rm Spec}(T):=\{p\in \mathscr{P}:\:(\exists t\geq 1)\; p|n_t\}$$
is finite. (We may think of $T$ as being ``close'' to an automorphism which is totally ergodic, in the sense, that most of \added{its} prime powers \deleted{is}\added{are} ergodic.)


%
%

\subsubsection{$\Z/2\Z$-extensions for which there are not too many roots. $k$-Morse cocycles}
Let $(X,T)$ be an odometer. Fix $k\geq1$ and assume that $k|n_t$, $t\geq 1$.
\begin{Def}\em  A cocycle $\phi:X\to\Z/2\Z$ is said to be a {\em $k$-Morse cocycle} if, for each $t\geq1$, $\phi$ is constant on each level $D^{(t)}_j$ except for
$D^{(t)}_{\frac ik n_t-1}$ for $i=1,2,\ldots, k$.
\end{Def}

{\em Morse cocycles} are, by definition, 1-Morse cocycles.
The skew products determined by $T_\phi$, where $\phi$ is a Morse cocycle, correspond to subshifts given by generalized Morse sequences~\eqref{eq:morse}. Moreover (see \cite{Le}), for each $t\geq1$ and $i=0,\ldots,n_t-2$, we have
\beq\label{eq:formcoc}
\phi|_{D^{(t)}_i}=\widehat{c}_t(i),
\eeq
where $c_t:=b^0\times\ldots \times b^{t-1}$ and then inductively
\beq\label{eq:formcoc1}
\phi|_{D^{(t+1)}_{in_t}}=\widehat{C}_{t+1}(in_t)=b^{t+1}(i-1)+b^{t+1}(i)+c_t(n_t-1),\;i=1,\ldots,\la_{t+1}-1.\eeq
The Morse sequence for which $b^t=01$ for each $t\geq0$ is the classical Thue-Morse sequence (e.g.\ \cite{Al-Sh}). More generally, the Morse sequences for which $b^t\in\{00,01\}$ with infinitely many blocks equal to $01$ are precisely Kakutani sequences, e.g.\ \cite{Kw}. Note that in \added{the} case of Kakutani sequences, for the corresponding odometer $(X,T)$, we have ${\rm Spec}(T)=\{2\}$, so we deal with small rational spectrum.

A similar situation arises when we consider the subshift given by the Rudin-Shapiro sequences (e.g.\ \cite{Al-Sh}). Indeed, in this case $\phi$ \deleted{are}\added{is a} $2$-Morse cocycle\deleted{s}, see  \cite{Le} for more details.

\begin{Def} \em Assume that $T$ has small rational spectrum and $\phi$ is a $k$-Morse cocycle. We say that
$\phi$ is {\em probabilistic} if there exists $\eta>0$ such that  for infinitely many $t\geq1$ the conditional distribution of $0$ on each level
$D^{(t)}_{\frac ik n_t-1}$ ($i=1,\ldots,k$)  is  between $\eta$ and $1-\eta$. We will say that $\phi$ satisfies PC (the probabilistic condition).\end{Def}

\begin{Remark}\em All Kakutani sequences satisfy PC in the following sense. First, notice that by introducing ``parentheses''
$$
x=(b^0\times\ldots\times  b^{i_1-1})\times(b^{i_1}\times\ldots\times b^{i_2-1})\times\ldots=\ov{b}^0\times\ov{b}^1\times\ldots$$
we can obtain a new representation of a Morse sequence $x$, in which the corresponding odometer $(\ov{X},\ov{T})$ is given by a subsequence of $(n_t)$, and we look at the Morse cocycle $\phi$ only along this subsequence.

Now, take any $x=b^0\times b^1\times\ldots$, where $b^t=00$ or $01$ (with infinitely many $b^t$ equal to $01$). Then introduce ``parentheses'' putting together
$01\times 01=0110$ and $00\times 01=0011$. Now, for the corresponding ${\ov{b}^t}'s$, we have  $\widehat{\ov{b}^t}$ equal either to  $101\ast$ or $010\ast$, so \deleted{the} PC is satisfied (cf.~\eqref{eq:formcoc1}).\end{Remark}

We can easily generalize this argument to obtain the following result.

\begin{Prop}\label{p:PCbounded} All (continuous) Morse sequences $x=b^0\times b^1\times\ldots$ with bounded lengths of blocks yield $\phi$ satisfying \deleted{the} PC.
\end{Prop}
\begin{proof} By introducing parentheses, if necessary, we can assume that $3\leq|b^t|\leq C$, $t\geq0$. Then if we see infinitely many blocks different from: $0\ldots0$, $01\ldots01$ or $01\ldots010$, then we are done. If not then if we have infinitely many blocks $0\ldots0$ then
$$
0\ldots 0\times 01\ldots01, \;\;\mbox{or}\;\;0\ldots 0\times 01\ldots010$$
yield also ``good'' blocks, so by introducing more parentheses, we obtain a new representation which is good, by looking at the last positions of appearances of $0\ldots0$. If not, then assume that we have infinitely many blocks $01\ldots01$. Then both blocks
$$
01\ldots 01\times 01\ldots01\text{ or }01\ldots 01\times 01\ldots010$$
are ``good'', and we are done since otherwise, starting from some place, we must have $b^t=01\ldots010$ which means that $x$ is periodic.
 \end{proof}

\begin{Remark}\label{r:RS17}\em
For Rudin-Shapiro sequences, it follows from \cite{Le} that the corresponding 2-Morse cocycles satisfy \deleted{the} PC.\end{Remark}

\begin{Th}\label{t:a1} Assume that an odometer $(X,T)$ has small rational spectrum, $\phi:X\to\Z/2\Z$ is a $k$-Morse cocycle which satisfies \deleted{the} PC. Then for all prime numbers $r\neq s$ sufficiently large, $(T_\phi)^r$ has no $s$th root. In particular, $(T_\phi)^r$ and $(T_\phi)^s$ are not isomorphic. In fact, the only ergodic joinings between $(T_\phi)^r$ and $(T_\phi)^s$ are the relatively independent extensions of isomorphisms between $T^r$ and $T^s$.\footnote{We recall that $T^r$ is isomorphic to $T$ whenever $T^r$ is ergodic. The last assertion in the theorem means that the assumptions of Theorem~\ref{t:SMext} are satisfied.}
\end{Th}
\begin{proof} Assume that $(T_\phi)^r$ and $(T_\phi)^s$ are  isomorphic (we assume that $r,s$ are coprime with $k$). Then $(T_\phi)^r$ has an $s$-th root which is in the centralizer of $(T_\phi)^r$. But, by Lemma 4.3 in \cite{Fe-Ku-Le-Ma} it follows that $C(T_\phi)=C((T_\phi)^r)$, hence this $s$-th root is in $C(T_\phi)$. It has to be of the form $\widetilde{W}=W_\xi$, where $(\widetilde{W})^s=(T_\phi)^r$. It follows that $W=T^{r/s}$. In other words, we can lift the rotation by $r/s$ to $C(T_\phi)$. Hence
\beq\label{a3}
\phi\circ W-\phi=\xi\circ T-\xi.\eeq

Recall that at the stage $t$, $W$ is represented by $b_t$ (i.e.\ on the tower $D^t$, $W$ acts as $T^{b_t}$).
In view of~\eqref{a6},
\beq\label{a4}\left\|(n_t-b_t)-i\frac{n_t}k\right\|\geq\eta n_t\eeq
uniformly in $t$ and $i\in\{0,\ldots,k-1\}$.

Fix $\vep>0$ small. Then for $t$ large enough the function $\xi$ will be
well approximated by the levels of $D^{(t)}$. Hence for $(1-\vep)n_t$ levels of $D^{(t)}$ we will have that $\xi$ is up to a set of conditional measure $\vep$ constant. In what follows we will speak about $\xi$ being $\vep$-constant on a level.

Consider  $D^{(t)}_{b_t}=T^{b_t}D^{(t)}_0$. In view of~\eqref{a4} there exists $0\leq i\leq k-1$ such that $$
\frac ik n_t<b_t< \frac{i+1}k n_t$$ with
$b_t - \frac ik n_t\geq \eta n_t$ and $\frac{i+1}k n_t -b_t|\geq \eta n_t$. It follows that we can find $0\leq\ell<\frac{1}k n_t$~\footnote{Note that $0<\frac{i+1}k n_t-b_t<\frac{1}k n_t$.}, $\ell<\frac{i+1}k n_t-b_t$ such that
$\xi|_{D^{(t)}_\ell}$ is $\vep$-constant. Moreover $\phi|_{WD^{(t)}_\ell}$ and $\phi|_{D^{(t)}_\ell}$ are also constant, and~\eqref{a3} makes $\xi|_{D^{(t)}_{\ell+1}}$ $\vep$-constant either with the same distribution as $\xi|_{D^{(t)}_\ell}$ or by replacing $0$ by $1$ and vice versa. We now repeat the same argument with $\ell$ replaced by $\ell+1$. We keep going in the same manner and we obtain consecutive levels on which $\xi$ is $\vep$-constant until $\ell$ reaches the value $\frac{i+1}k n_t-b_t$. Now, $\phi\circ W$ will have the same distribution as that of $\phi$ on $D^{(t)}_{\frac{i+1}k n_t-1}$, while $\phi$ on $D^{(t)}_{\frac{i+1}k n_t-b_t}$ is still constant and $\xi$ is $\vep$-constant. It follows that up to $\vep$ the conditional distribution of $\xi$ on $TD^{(t)}_{\frac{i+1}k n_t-b_t}$ is that of $\phi$ on $D^{(t)}_{\frac{i+1}k n_t}$ or its ``mirror''. When we consider the next step $\phi\circ W$ and $\phi$ will be again constant, so $\xi$ on the next level will have the same distribution as on
the previous level (or
its ``mirror''). This will be continued for $\eta n_t$ levels and because $\phi$ satisfies \deleted{the} PC, $\xi$ cannot be measurable, a contradiction.

The last assertion follows from non-isomorphism of the powers and Corollary~4.7 in \cite{Fe-Ku-Le-Ma}.
\end{proof}

\begin{Remark}\em For Morse cocycles, the above result can also be deduced from a theorem proved by Kwiatkowski and Rojek \cite{Kw-Ro} about the centralizer of Morse subshifts.
\end{Remark}

Now,  we obtain that whenever the (uniquely ergodic) subshift $(X_x,S)$ determined by a $k$-Morse sequence $x\in\{0,1\}^{\N}$ is given by \deleted{and} a $k$-Morse cocycle satisfying PC, then  the
uniquely ergodic system $(X_x,S)$ is an extension of $(X_{\widehat{x}},S)$ for which the assumptions of Theorem~\ref{t:SMext} hold and therefore, it satisfies the strong MOMO property.
Using~\ref{thmAPm} (and Remark~\ref{r:RS17}), we hence obtain the following.

\begin{Cor}\label{c:Kak+RS} In each uniquely ergodic model of the system determined by a Kakutani sequence\footnote{It has been already known that the subshift determined by any Kakutani sequence is M\"bius disjoint \cite{Bo,Gr,Fe-Ku-Le-Ma,Ve}.} (in particular, by the Thue-Morse sequence) we have the strong MOMO property (relatively to $\bfu$ satisfying~\eqref{e:mrt}). The same result holds for any Rudin-Shapiro sequence.\end{Cor}

\subsection{Substitutions of constant length}
Let $A$ be a finite alphabet, $\# A\geq2$. Let $q\geq2$ be a fixed integer and $\theta:A\to A^q$ be a primitive aperiodic substitution of constant length~$q$ \cite{Qu}. Recall that
$\theta$ is extended to a morphism of the monoid $A^*$  by the formula
$$
\theta(a_0\cdots a_{\ell-1}):=\theta(a_0)\cdots\theta(a_{\ell-1}).
$$
Similarly, we can extend $\theta$ to a map defined on $A^{\N}$.
We denote by $X_\theta$ the two-sided associated subshift:
\begin{multline*}
  X_\theta:=\Bigl\{x=(x(n),\ n\in\Z)\in A^\Z: \\
  \forall m\le n, \exists t\ge0, \exists a\in A, x(m,n)\text{ is a subblock of }\theta^t(a)\Bigr\}.
\end{multline*}
Let $S$ be the shift map on $X_\theta$. We recall that $(X_\theta,S)$ is uniquely ergodic, and we denote by $\mu_\theta$ the unique $S$-invariant probability measure on $X_\theta$.
To each $\theta$ we can associate the {\em column number}.
\begin{Def}[Kamae~\cite{Kamae72}]
\em
  The \emph{column number} of the substitution $\theta$ is the number
  \[
    c(\theta) := \min_{t\ge1} \min_{0\le \ell \le q^t-1} \#\{\theta^t(a)(\ell): a\in A\}.
  \]
\end{Def}
If by $X_q$ we denote the odometer determined by $n_t:=q^t$, $t\geq0$, then $(X_q,T)$ is the maximal equicontinuous factor of $(X_\theta,S)$ and moreover (see \cite{Kamae72,Qu})
\beq\label{eq:cext}
\mbox{$(X_\theta,\cb(X_\theta),\mu_\theta,S)$ is an a.e.\ $c$-extension of $(X_q,\cb(X),\mu_{X_q},T)$.}\eeq

\subsubsection{Bijective substitutions}
A substitution $\theta\colon A\to A^q$ (as above) is called {\em bijective} if the map
$$
\tau_i(a):=\theta(a)(i),\;a\in A$$
is a bijection of $A$ for each $i=0,1,\ldots,q-1$.
Using the notion of group substitution, it is implicitly proved in \cite{Fe-Ku-Le-Ma}\footnote{This follows from the proofs of Proposition 4.5 and Theorem 5.4 (see also Proposition 4.2) in \cite{Fe-Ku-Le-Ma}.} that the assumptions of Theorem~\ref{t:SMext} are satisfied. It follows that:

\begin{Cor}\label{p:SMbij} If $\theta$ is bijective then the strong MOMO property (relatively to \added{a bounded multiplicative} $\bfu$ satisfying~\eqref{e:mrt}) is satisfied in each uniquely ergodic model of $(X_\theta,\cb(X_\theta),\mu_\theta,S)$.\end{Cor}

\subsubsection{The synchronized case}
For one more case of substitutions the assertion of Proposition~\ref{p:SMbij} easily holds. Namely, this is the case when the column number $c(\theta)=1$. Indeed, in this case, by~\eqref{eq:cext}, the factor map
$$(X_\theta,\cb(X_\theta),\mu_\theta,S) \to (X_q,\cb(X_q),\mu_{X_q},T)$$
is a.e.\ 1-1. It easily follows that $(X_\theta,S)$ is a uniquely ergodic model of the odometer $(X_q,T)$, hence the result follows from Corollary~\ref{c:SMds}.

It is well-known (Cobham's theorem) that fixed points of substitutions of constant length are in one-to-one correspondence with automatic sequences, i.e.\ sequences generated by deterministic complete automata \cite{Qu}. Those automatic sequences which correspond to synchronized automata are called {\em synchronized}, and the substitutions with trivial column number are in 1-1 correspondence with synchronized automatic sequences \cite{Mu}. An independent proof of M\"obius disjointness in the synchronized case has been done in \cite{De-Dr-Mu}.

As all subshifts given by substitutions of constant length are M\"obius disjoint by a recent result of M\"ullner \cite{Mu}, it is natural to ask:
\begin{Question}
Is it true that all subshifts generated by substitutions of constant length satisfy the (strong) MOMO property (relatively to $\mob$)?
\end{Question}

\begin{Remark}
It seems that the main problem to obtain Corollary~\ref{p:SMbij} without any restriction on $\theta$ is a full description of the cocycle $\va\colon X_q\to \cs(\{0,1,\ldots,c(\theta)-1)\})$ which is behind
the statement~\eqref{eq:cext}. \end{Remark}
\small

\bibliography{MOMO}

\end{document}